\documentclass[11pt, oneside]{amsart}
\usepackage{amsfonts, amstext, amsmath, amsthm, amscd, amssymb}
\usepackage{manfnt}
\usepackage{mathrsfs}
\usepackage{url}
\usepackage[all]{xypic}

\usepackage{multido}

\usepackage[dvips]{graphicx}
\usepackage{enumitem}

\usepackage{marginnote}

\usepackage[paper=a4paper, text={138mm,208mm},centering]{geometry}

\usepackage{graphics, pinlabel, color}
\usepackage{comment}

\usepackage[all,graph]{xy}
\usepackage{psfrag}

\renewcommand{\setminus}{{\smallsetminus}}

\newcommand{\lmfrac}[2]{\mbox{\small$\displaystyle\frac{#1}{#2}$}}   
\newcommand{\smfrac}[2]{\mbox{\footnotesize$\displaystyle\frac{#1}{#2}$}} 

\newcommand{\smsum}[2]{\mbox{\footnotesize$\displaystyle\sum\limits_{#1}^{#2}$}}
\newcommand{\tmsum}[2]{\mbox{$\textstyle \sum\limits_{#1}^{#2}$}}

\newcommand{\smcup}[2]{\mbox{\footnotesize$\displaystyle\bigcup\limits_{#1}^{#2}$}}
\newcommand{\tmcup}[2]{\mbox{$\textstyle \bigcup\limits_{#1}^{#2}$}}

\newcommand{\smoplus}[2]{\mbox{\footnotesize$\displaystyle\bigoplus\limits_{#1}^{#2}$}}

\newcommand{\smprod}[2]{\mbox{\footnotesize$\displaystyle\prod\limits_{#1}^{#2}$}}
\newcommand{\tmprod}[2]{\mbox{$\textstyle \prod\limits_{#1}^{#2}$}}

\def\alphas{{\widehat{\alpha}}}
\newcommand{\bp}{\begin{pmatrix}}
\newcommand{\ep}{\end{pmatrix}}
\newcommand{\be}{\begin{equation}}
\newcommand{\ee}{\end{equation}}
\newcommand{\ol}[1]{\overline{#1}}

\numberwithin{equation}{section}

\theoremstyle{plain}
\newtheorem{theorem}[equation]{Theorem}
\newtheorem{lemma}[equation]{Lemma}
\newtheorem{proposition}[equation]{Proposition}

\newtheorem{algorithm}[equation]{Algorithm}
\newtheorem{construction}[equation]{Construction}

\newtheorem{thm}[equation]{Theorem}

\newtheorem{prop}[equation]{Proposition}
\newtheorem*{claim*}{Claim}

\theoremstyle{definition}
\newtheorem{example}[equation]{Example}
\newtheorem{remark}[equation]{Remark}

\newtheorem{definition}[equation]{Definition}

\newtheorem{conventions}[equation]{Conventions}

\newtheorem{defn}[equation]{Definition}

\numberwithin{equation}{section}

 \newtheoremstyle{TheoremNum}
        {}{}              
        {\itshape}                      
        {}                              
        {\bfseries}                     
        {.}                             
        { }                             
        {\thmname{#1}\thmnote{ \bfseries #3}}
\theoremstyle{TheoremNum}

\def\N{\mathbb N}
\def\Z{\mathbb Z}
\def\R{\mathbb R}
\def\Q{\mathbb Q}
\def\C{\mathbb C}

\def\F{\mathbb{F}}
\def\wt#1{\widetilde{#1}}

\def\sm{\setminus}

\def\a{\alpha}

\def\g{\gamma}

\def\toiso{\xrightarrow{\simeq}}
\def\zp{\mathbb{Z}[\pi]}
\def\ll{\langle}

\def\rr{\rangle}

\def\zt{\Z[t,t^{-1}]}

\def\ft{\F[t,t^{-1}]}
\def\bp{\begin{pmatrix}}
\def\ep{\end{pmatrix}}
\def\ba{\begin{array}}
\def\ea{\end{array}}
\def\bn{\begin{enumerate}}
\def\en{\end{enumerate}}

\def\zpx{\Z[\pi]}
\def\op{\operatorname}
\def\BS{\op{BS}}
\def\PD{\op{PD}}
\def\ev{\op{ev}}

\DeclareMathOperator\Ext{Ext}

\DeclareMathOperator\Hom{Hom}
\DeclareMathOperator\Aut{Aut}
\DeclareMathOperator\Id{Id}
\DeclareMathOperator\ord{ord}

\DeclareMathOperator\GL{GL}
\DeclareMathOperator\Bl{Bl}
\DeclareMathOperator\im{im}
\DeclareMathOperator\ab{ab}

\DeclareMathOperator{\Wr}{Wr}
\DeclareMathOperator{\sign}{sign}

\def\tpm{[t,t^{-1}]}
\def\bla{\Bl^\a}

\def\ftk{\F(t)^k}

\def\hom{\op{Hom}}
\def\ol{\overline}
\def\PD{\op{PD}}
\def\ev{\op{ev}}
\def\what{\widehat}

\newcommand{\eps}{\varepsilon}


\begin{document}

\title[Twisted Blanchfield pairings of knots]{Symmetric chain complexes, twisted Blanchfield pairings, and knot concordance}

\author{Allison N.~Miller}
\address{Department of Mathematics, University of Texas, Austin, USA}
\urladdr{http://www.ma.utexas.edu/users/amiller/index.html}
\email{amiller@math.utexas.edu}

\author{Mark Powell}
\address{Department of Mathematics, Durham University, United Kingdom}
\email{mark.a.powell@durham.ac.uk}

\begin{abstract}
We give a formula for the duality structure of the 3-manifold obtained by doing zero-framed surgery along a knot in the 3-sphere, starting from a diagram of the knot. We then use this to give a combinatorial algorithm for computing the twisted Blanchfield pairing of such 3-manifolds. With the twisting defined by Casson-Gordon style representations, we use our computation of the twisted Blanchfield pairing to show that some subtle satellites of genus two ribbon knots yield non-slice knots.  The construction is subtle in the sense that, once based, the infection curve lies in the second derived subgroup of the knot group.
\end{abstract}

\maketitle

\section{Introduction}

This article has three parts.  The first part describes the \emph{symmetric Poincar\'e chain complex} of the 3-manifold $M_K$ obtained by doing 0-framed Dehn surgery on $S^3$ along a knot $K \subset S^3$.  The second part gives an algorithm to compute the \emph{twisted Blanchfield pairing} of $M_K$ with respect to a representation of its fundamental group.  Finally, we give an application of our ability to implement this computation to \emph{knot concordance}.

\subsection{The symmetric chain complex of the zero surgery}

Let $\pi$ be a group and let $n \in \mathbb{N}_0$.  Roughly speaking, an $n$-dimensional symmetric chain complex \cite{Ranicki3} $(C_*,\Phi)$ is a chain complex $(C_*,\partial)$ of free finitely generated $\Z[\pi]$-modules, together with a chain map $\Phi_0 \colon C^{n-*} \to C_*$, a chain homotopy $\Phi_1 \colon \Phi_0 \sim \Phi_0^*$, together with a sequence of higher chain homotopies $\Phi_{i+1} \colon \Phi_{i} \sim (-1)^i \Phi_i^*$, for $i=1,\dots,n-1$.  The r\^ole in this article of higher homotopies will be peripheral.

An $n$-dimensional manifold $M$ with $\pi_1(M) =\pi$ gives rise to an $n$-dimensional symmetric chain complex over $\Z[\pi]$.  In this case the maps $\varphi_0$ induce the Poincar\'e duality isomorphisms
\[-\cap [M] \colon H^{n-r}(M;\Z[\pi]) \to H_r(M;\Z[\pi]).\]
More generally, an arbitrary symmetric complex is called \emph{Poincar\'e} if the maps $\Phi_0$ constitute a chain equivalence.    The symmetric chain complex of a manifold contains the maximal data that the manifold can give to homological algebra via a handle or CW decomposition.

The first part of this paper, comprising Sections \ref{section:Poincare-cxs-definitions} and \ref{section:example-knot-exteriors}, gives a procedure to explicitly write down the $3$-dimensional symmetric Poincar\'e chain complex of the zero-framed surgery manifold $M_K$ of an oriented knot $K \subset S^3$.

\begin{algorithm}\label{algorithm:symm-chain-cx}
  We describe a combinatorial algorithm that takes as input a diagram of an oriented knot $K$ and produces
 a symmetric chain complex $(C_*,\Phi)$ of the zero-framed surgery on $K$ with coefficients in $\Z[\pi_1(M_K)]$, with explicit formulae for the boundary maps $\partial \colon C_i \to C_{i-1}$ and the symmetric structure maps $\Phi_0 \colon C^{3-r} \to C_r$.
\end{algorithm}

This is based on a precise understanding of a handle decomposition of $M_K$ (Construction~\ref{Thm:includingboundary}), from which we exhibit, in Theorem~\ref{Thm:mainchaincomplex}, a cellular chain complex for (a space homotopy equivalent to) $M_K$ with coefficients in $\pi_1(M_K) =: \pi$.
The novelty is the use in Section~\ref{section:trotters-formulae} of formulae of Trotter~\cite{Trotter} to produce a diagonal chain approximation map
\[\Delta_0 \colon C_*(M_K;\Z) \to C_*(M_K;\Z[\pi]) \otimes_{\Z[\pi]} C_*(M_K;\Z[\pi]). \]
The image $\Delta_0([M_K])$ of a fundamental class $[M_K] \in C_3(M_K;\Z)$ under $\Delta_0$ gives rise to the $\Phi_0$ maps, under the identification $(C_* \otimes C_*)_3 \cong \Hom(C^{3-*},C_*)$, where $C_* = C_*(M_K;\Z[\pi])$.

\subsection{The twisted Blanchfield pairing}

Let $R$ be a commutative principal ideal domain with involution, and let $Q$ be its quotient field.  Let $\alpha \colon \pi_1(M_K) \to U(R^k)$ be a unitary representation of the fundamental group $\pi := \pi_1(M_K)$ of~$M_K$.  This makes $R^k$ into an $(R,\Z[\pi])$-bimodule, using the right action of $U(R^k)$ on $R^k$ represented as row vectors.
We can use this representation to define the twisted homology $H_*(M_K;R^k_{\alpha})$ as follows.  Start with the chain complex $C_*(M_K;\Z[\pi])$ and tensor over the representation to obtain $R^k \otimes_{\Z[\pi]} C_*(M_K;\Z[\pi])$.  The \emph{homology of $M_K$ twisted over $\alpha$} is the homology $H_*(R^k \otimes_{\Z[\pi]} C_*(M_K;\Z[\pi]))$.
The $R$-torsion submodule of an $R$-module $P$ is $TP := \{p\in P \mid rp=0 \text{ for some } r \in R\sm \{0\} \}$.
The \emph{twisted Blanchfield pairing}
\[\Bl^{\a} \colon TH_1(M_K;R^k_\a) \times TH_1(M_K;R^k_\a) \to Q/R\]
is a nonsingular, hermitian, sesquilinear form defined on the $R$-torsion submodule of the first homology.

The precise definition of the twisted Blanchfield pairing can be found in Section~\ref{section:defn-TBF}, but we give an outline here.
Start with a CW decomposition of $M_K$.  We want to compute the pairing of two elements $[x],[y] \in TH_1(M_K;R^k_\a)$, represented as 1-chains $x$ and $y$ in the cellular chain complex $C_1(M_K;R^k_\a)$ of $M_K$ with coefficients in $R^k_{\a}$.  Find the Poincar\'e dual $[v] \in TH^2(M_K;R^k_\a)$ of $[x]$, represented by a $2$-cochain $v \in C^2(M_K;R^k_{\a})$ such that $v \cap [M_K] =x$.  Since $[v]$ lies in the $R$-torsion subgroup, there exists $r \in R$ and $w \in C^1(M_K;R^k_{\a})$ such that $\partial^*(w)=v$.  We then pair $w$ and $y$ and divide by $r$, to obtain \[\Bl^{\a}([x],[y]) = w(y)/r.\]  This is an element of $Q$ whose image in the quotient $Q/R$ is well-defined, being independent of the choices of chains $x,y,$ and $w$ and of the element $r \in R$.

This procedure can be explicitly followed using the data of the symmetric chain complex of $M_K$.  In Section~\ref{section:algebraic-defn-TBF} we give an algorithm to make this computation, and we implement this algorithm using Maple.  This enables us to explicitly compute the twisted Blanchfield pairing of a pair of elements of $TH_1(M_K;R^k_\a)$, at least for suitably amiable representations.

\begin{algorithm}\label{algorithm:twisted-Bl-pairing}
  We describe a combinatorial algorithm that takes as input a $3$-dimensional symmetric chain complex over $\Z[\pi]$, a unitary representation $\alpha \colon \pi_1(M_K) \to U(R^k)$, and two elements $x,y \in  TH_1(M_K;R^k_\a)$, and outputs the twisted Blanchfield pairing $\Bl^{\a}(x,y) \in Q/R$.
\end{algorithm}

\subsection{Constructing non-slice knots}

An oriented knot $K$ in $S^3$ is said to be a \emph{slice knot} if there is a locally flat proper embedding of a disc $D^2 \hookrightarrow D^4$, with the boundary of $D^2$ sent to $K \subset S^3$.  The set of oriented knots modulo slice knots inherits a group structure from the connected sum operation, called the knot concordance group and denoted by~$\mathcal{C}$.  Throughout the paper, for a submanifold $N \subset M$, let $\nu N$ denote a tubular neighbourhood of $N$ in $M$.  Note that the boundary $\partial(D^4 \sm \nu D^2)$ of the exterior of a slice disc is the zero-framed surgery manifold $M_K$.

We will construct new non-slice knots that lie in the kernel of Levine's~\cite{Levine:1969-1} homomorphism $\mathcal{L} \colon \mathcal{C} \to \mathcal{AC}$ to the \emph{algebraic concordance group} $\mathcal{AC} \cong \Z^\infty \oplus \Z_2^{\infty} \oplus \Z_4^{\infty}$ of Seifert forms modulo metabolic forms.  Here a Seifert form is \emph{metabolic} if there is a half-rank summand on which the form vanishes.

To construct our non-slice knots we will use a \emph{satellite construction}.
Let $K$ be an oriented knot in $S^3$, let $\eta \subset S^3 \sm \nu K$ be a simple closed curve in $S^3 \sm \nu K$, which is unknotted in $S^3$, and let $J \subset S^3$ be another oriented knot.
The knot $K$ will be referred to as the pattern knot, $\eta$ as the infection curve (or axis), and $J$ will be referred to as the infection (or companion) knot.

  Consider the $3$-manifold \[\Sigma := S^3 \sm \nu \eta \cup_{\partial \mathrm{cl}(\nu \eta)} S^3 \sm \nu J,\]
  where the gluing map identifies the meridian of $\eta$ with the zero-framed longitude of $J$, and vice versa. The 3-manifold $\Sigma$ is diffeomorphic to $S^3$, via an orientation preserving diffeomorphism that is unique up to isotopy.  The image of $K \subset S^3 \sm \nu \eta$ under this diffeomorphism is by definition the satellite knot $K_{\eta}(J)$; this operation of altering~$K$ by~$J$ is called the \emph{satellite construction} or \emph{genetic infection}.
In our constructions, we will start with a slice knot $K$, and for suitable $\eta$ and $J$ we will show that $K_{\eta}(J)$ is not slice.

Our non-slice knots will be produced using a single explicitly drawn curve $\eta \in \pi_1(S^3 \sm \nu K)^{(2)}$, the second derived subgroup of the knot group.
Here the \emph{derived series} of a group $\pi$ is defined via $\pi^{(0)}:= \pi$ and $\pi^{(i+1)} :=[\pi^{(i)},\pi^{(i)}]$, the smallest normal subgroup containing $ghg^{-1}h^{-1}$ for all $g,h \in \pi^{(i)}$.

In usual constructions of this sort, one often has $\eta \in \pi_1(S^3 \sm \nu K)^{(1)}$, the commutator or first derived subgroup of the knot group.  For examples of non-slice knots arising from satellite constructions, see the use of Casson-Gordon invariants~\cite{Casson-Gordon:1978-1},~\cite{Casson-Gordon:1986-1} in~\cite{Gilmer:1983-1},~\cite{Livingston:1983-1}, \cite{Gilmer-Livingston:1992-1}, \cite{Livingston:2002-1} and \cite{Livingston-2002-2}, and the use of $L^{(2)}$-signature techniques in \cite{Cochran-Orr-Teichner:1999-1}, \cite{Cochran-Orr-Teichner:2002-1}, \cite{Cochran-Kim:2004-1}, \cite{Cochran-Harvey-Leidy:2009-1}, \cite{Cochran-Harvey-Leidy:2008-2}, \cite{Cha-Orr:2009-01}, \cite{Cha:2014-1} and~\cite{Franklin:2013}.
We also allow pattern knots $K$ of arbitrary genus, whereas in many of the $L^{(2)}$ papers listed above, the pattern knots were often genus one.  Our approach generalises the example from \cite[Section~6]{Cochran-Orr-Teichner:1999-1}, and indeed in Section~\ref{section:examples} we reprove that the knot considered there is not slice.   In \cite{Cochran-Kim:2004-1}, pattern knots were genus two and higher, but they used multiple infection curves.
Here is a discussion of the previous literature and its relation to our knots.  We are grateful to Taehee Kim for sharing his perspective.

\begin{enumerate}
  \item In \cite{Cochran-Harvey-Leidy:2009-1} and \cite{Cha:2014-1}, non-slice knots were constructed by iterated satellite constructions.  Examples were given with a single infection curve.
       Start with a ribbon knot $R$ and an infection curve $\eta_1$ in $\pi_1(S^3 \sm \nu R)^{(1)}$.  Then infect $R$ with itself to obtain the satellite knot $R(R,\eta_1)$.   Now to construct non-slice knots, \cite{Cochran-Harvey-Leidy:2009-1} and \cite{Cha:2014-1} infect this using a curve that lies in the second derived subgroup $\pi_1(S^3 \sm \nu R(R,\eta_1))^{(2)}$.
       However in these constructions the infection curves lie in the first derived subgroup of each of the building pieces of the iterated satellite construction, and the non-triviality of these curves in a slice disc complement is detected by the classical Blanchfield pairing of each piece.
  \item In \cite{Cochran-Kim:2004-1}, the infection curves arise as commutators of generators of $\pi_1(F)$, where $F$ is a minimal genus Seifert surface for the pattern knot.  They can be drawn explicitly, although this would be quite laborious.
  \item In the current paper, we obtained our examples by drawing a likely-looking curve, and then checking by computation with the twisted Blanchfield pairing, as explained below, that infection gives rise to a non-slice knot.  In \cite{Cochran-Kim:2004-1,Cochran-Harvey-Leidy:2009-1,Cha:2014-1}, they found homology classes that work to produce non-slice knots from the algebra of higher order Alexander modules over non-commutative rings. One can draw representative infection curves in a knot diagram.  In this previous work, the emphasis was on finding non-slice knots with certain properties relating to the solvable filtration.  In the present work, we aim to provide a new tool to detect non-slice knots.
      The fact that we work with commutative rings makes the twisted Blanchfield pairing a particularly useful computational tool.
\end{enumerate}

The key to our approach is to show that the infection curve $\eta$, when thought of as an element of $\pi_1(M_{K_{\eta}(J)})$, represents a nontrivial element of $\pi_1(D^4 \sm \nu D^2)$, for \emph{any} possible slice disc $D^2 \subset D^4$ for $K_{\eta}(J)$.  We will achieve this using the twisted Blanchfield pairing, as we explain next.

For a knot $K$, let $\Sigma_k(K)$ be the $k$-fold branched cover of $S^3$ branched along $K$.  Recall that there is a nonsingular symmetric linking pairing $\lambda_k \colon H_1(\Sigma_k(K);\Z) \times H_1(\Sigma_k(K);\Z) \to \Q/\Z$, and that a metaboliser $P \subseteq H_1(\Sigma_k(K);\Z)$ is a submodule of square-root order on which the linking pairing vanishes.
In Section \ref{section:reps} we will associate, to a knot $K$ and a metaboliser $P$, along with some auxiliary choices, a unitary \emph{Casson-Gordon type representation} $\a_P \colon \pi_1(M_K) \to U(k,\ft)$. Here $\mathbb{F} := \Q(\zeta_q)$ with $\zeta_q$ a $q$-th root of unity, for $q$ a prime power.  With such representations, the twisted Blanchfield pairing gives rise to the following slice obstruction theorem, the full version of which appears as Theorem~\ref{thm:slice}.

\begin{theorem}\label{thm:metabolic-TBP-intro}
Let $K$ be an oriented slice knot with slice exterior $W:= D^4 \sm \nu D^2$.
 Then for any prime power $k$, there exists a metaboliser $P$ of $\lambda_k$
such that for any Casson-Gordon type representation $\a_P \colon \pi_1(M_K) \to U(k,\Q(\zeta_q)[t^{\pm 1}])$ corresponding to $P$, there is a prime power $q'$ with $q \mid q'$ such that the twisted Blanchfield pairing $\Bl(i \circ \a_P)$ is metabolic with metaboliser \[\ker\left( TH_1(M_K;\ft^k_{\a_P}) \to TH_1(W;\ft^k_{\a_P})\right),\] where $i$ is the inclusion on the level of unitary groups corresponding to the inclusion $\Z_q \hookrightarrow \Z_{q'}$ and $\mathbb{F} = \Q(\zeta_{q'})$.
\end{theorem}

The extension from $q$ to $q'$ is potentially necessary in order to extend the representation over the slice exterior~$W$.
This theorem recovers the twisted Fox-Milnor condition of \cite{Kirk-Livingston:1999-2}, that twisted Alexander polynomials of slice knots factor as a norm (Lemma~\ref{lem:alexnorm}).
In order to use this theorem to go beyond the results of Kirk and Livingston, we use the following obstruction theorem, which is based on ideas of Cochran, Harvey and Leidy~\cite{Cochran-Harvey-Leidy:2009-1}.  The version with full details appears below as Theorem~\ref{thm:sliceobstruction}; in particular the theorem will be generalised to obstruct 2.5-solvability (Definition~\ref{Defn:filtration}).  To state the theorem, we should recall the Tristram-Levine signature function: let $V$ be a Seifert matrix of a knot $J$, and then define
$\sigma_J \colon S^1 \to \Z$ by $\omega \mapsto \sign\big((1-\omega)V + \ol{(1-\omega)}V^T \big)$, thinking of $S^1 \subset \mathbb{C}$.  For an algebraically slice knot $J$, $\sigma_J$ is almost everywhere zero on $S^1$.  Thus $\int_{S^1} \sigma_J(\omega)\, d \omega =0$ for algebraically slice knots.

\begin{theorem}\label{thm:main-slice-obstruction-L2-intro}
Let $R$ be a slice knot and let $\eta \in \pi_1(S^3 \sm \nu R)^{(2)}$.
Suppose that there is some prime power $k$ such that for each metaboliser $P$ for the linking form $\lambda_k(R)$, there is some Casson-Gordon type representation $\a_P$ corresponding to $P$ such that
\[ \Bl_{M_R}^{\alpha_P}(\eta, \eta) \neq 0 \text{ in } \mathbb{F}(t)/ \ft.\]
Then there is a constant $C_R>0$, depending only on the knot $R$, such that if $J$ is a knot with $\left| \int_{S^1}\sigma_J(\omega)\, d \omega \right|>C_R$,
then $K := R_{\eta}(J)$ is not slice.
\end{theorem}

The idea behind the proof is that the Blanchfield pairing condition guarantees that $\eta$ does not live in any metaboliser, and therefore does not lie in the kernel of the map induced on fundamental groups by the inclusion of the zero surgery into the slice disc exterior.  Using this ``robustness'' of $\eta$ together with the condition on the integral of the Tristram-Levine signatures of $J$ from the theorem, one can show that the $L^{(2)}$ $\rho$-invariant of $M_{R_{\eta}(J)}$ must be large, obstructing  $R_{\eta}(J)$ from being slice.   Connoisseurs might enjoy the novel use of a mixed coefficient derived series in Proposition~\ref{prop:keyproposition}.  We present some examples of the use of Theorem~\ref{thm:main-slice-obstruction-L2-intro}; details appear in Section~\ref{section:examples}.

\begin{proposition}
  Suppose the knot $J$ is such that $\left| \int_{S^1}\sigma_J(\omega)\, d \omega \right|> 10^{10}$
  (for example, $J$ is a connected sum of $10^{10}$ right handed trefoils).  Let $(R,\eta)$ be one of the $($pattern knot, infection curve$)$ pairs from Figure~\ref{Fig:88example}, Figure~\ref{Fig:cotexample} or Figure~\ref{Fig:squareknotexample}.   Then $R_{\eta}(J)$ is not slice.
\end{proposition}

The number $10^{10}$ is a power of 10 guaranteed to overcome the universal Cheeger-Gromov bound \cite{Cheeger-Gromov:1985-1}. In particular the explicit upper bound from \cite{Cha:2016-CG-bounds} is less than $10^8$ times the crossing number of $R$, and the crossing number of $R$ in all our examples is less than $10^2$.

As alluded to above, our second example in Section~\ref{section:examples} is the knot constructed in \cite[Section~6]{Cochran-Orr-Teichner:1999-1}, which was the first example of a non-slice knot with vanishing Casson-Gordon invariants.  We give a simpler proof and more generally applicable proof than theirs that this knot is not slice.

\subsection*{Conventions.} Throughout the paper we assume that all manifolds are connected, compact and oriented, unless we say explicitly otherwise.

\subsection*{Acknowledgments.}
We thank Stefan Friedl for many valuable discussions and comments on the paper.  The middle third of this paper arose from a project began several years ago thanks to discussions of the second author with Stefan.
The first third of this paper is based on material from the 2011 University of Edinburgh PhD thesis of the second author, supervised by Andrew Ranicki.
We are indebted to Taehee Kim for an enlightening discussion on the relationship of our results with others in the literature.
We would also like to thank Jae Choon Cha, Anthony Conway, Chris Davis, Min Hoon Kim, Taehee Kim, Matthias Nagel and Patrick Orson for their interest and input.
We thank the anonymous referee for a careful reading and invaluable suggestions for improving the paper.
The authors are grateful to the Hausdorff Institute for Mathematics in Bonn, in whose excellent research atmosphere part of this paper was written.
The second author is supported by an NSERC Discovery Grant.

\section{Symmetric Poincar\'{e} complexes}\label{section:Poincare-cxs-definitions}

In this section we introduce some basic homological algebra definitions, including sign conventions, and we give the precise definition of a symmetric complex.
The material of this section is due to Ranicki, primarily~\cite{Ranicki3}, and the reader looking for more details is referred to there.

\subsection{Basic chain complex constructions and conventions}\label{section:basic-chain-cx-constructions}

 Recall that $R$ denotes a ring with involution.  By convention, chain complexes consist of left $R$-modules unless otherwise stated.  Given a chain complex $C$ of left $R$-modules, let $C^t$ be the chain complex of right $R$-modules obtained by converting each left module to a right module using the involution on $R$.  That is, for  $c \in C$ and $r \in R$ the  right action of $r$ on $c$ is given by $c \cdot r := \overline{r} c$.

\begin{definition}[Tensor chain complexes]\label{Defn:signsoontensor}
Given chain complexes $(C,d_C)$ and $(D,d_D)$ of finitely generated (henceforth f.g.) projective $R$-modules, form the tensor product chain complex $C \otimes_{R} D$ with chain groups:
\[(C^t \otimes_R D)_n := \smoplus{p+q=n}{} \, C^t_p \otimes_R D_q.\]
 The boundary map
\[d_{\otimes} \colon (C^t \otimes_R D)_n \to (C^t \otimes_R D)_{n-1}\]
is given, for $x \otimes y \in C^t_p \otimes_R D_q \subseteq (C^t \otimes_R D)_n$, by
\[d_{\otimes}(x \otimes y) = x \otimes d_D(y) + (-1)^q d_C(x) \otimes y.\]
\end{definition}

\begin{definition}[$\Hom$ chain complexes]\label{defn:hom-chain-complex}
Define the complex $\Hom_R(C,D)$ by
\[\Hom_R(C,D)_n := \smoplus{q-p=n}{}\,\Hom_R(C_p,D_q)\]
with boundary map
\[d_{\Hom} \colon \Hom_R(C,D)_n \to \Hom_R(C,D)_{n-1}\]
given, for $g \colon C_p \to D_q$, by
\[d_{\Hom}(g) = d_D g + (-1)^q g d_C.\]
\end{definition}

\begin{definition}[Dual complex]\label{defn:dual-complex}
The dual complex $C^*$ is defined as a special case of Definition~\ref{defn:hom-chain-complex} with $D_0 = R$ as the only non-zero chain group. Note that $D_0=R$ is also an $R$-bimodule.  Explicitly we define $C^r := \Hom_R(C_r,R)^t$, with boundary map
$\delta=d^*_C \colon C^{r-1} \to C^{r}$
defined as
$\delta(g) = g \circ d_C.$
Using that $R$ is a bimodule over itself, the chain groups of $C^*$ are naturally right modules. But we use the involution to make them into left modules, as described in Section~\ref{section:twistedhomology}: for $f \in C^*$ and $a \in R$, let $(a \cdot f)(x) := f(x) \ol{a}$.

The chain complex $C^{-*}$ is defined to be
\[(C^{-*})_r = C^{-r};\;\; d_{C^{-*}} = (d_C)^*=\delta.\]
Also define the complex $C^{m-*}$ by:
\[(C^{m-*})_r = \Hom_R(C_{m-r},R)\]
with boundary maps
\[\partial^* \colon (C^{m-*})_{r+1} \to (C^{m-*})_{r}\]
given by
\[\partial^* = (-1)^{r+1}\delta.\]
\end{definition}

Define the dual of a cochain complex (i.e.\ the double dual) to be $C^{**}:=(C^{-*})^{-*}$.
The next proposition allows us to identify a chain complex with its double dual; its proof is a straightforward verification.

\begin{proposition}[Double dual]
For a finitely generated projective chain complex $C_*$, there is an isomorphism
$C_* \xrightarrow{\simeq} C^{**}$ given by $x \mapsto (f \mapsto \ol{f(x)}).$
\end{proposition}

\begin{definition}[Slant map]\label{defn:slant-map}
The slant map is the isomorphism
\[\ba{rcl} \backslash \colon C^t \otimes_R C & \to & \Hom_R(C^{-*},C_*)\\
x \otimes y & \mapsto & \left(g \mapsto \overline{g(x)}y\right). \ea\]
\end{definition}

\begin{definition}[Transposition]
Let $C_*$ be a chain complex of projective left $R$-modules for a ring with involution~$R$.  Define the transposition map
\[\ba{rcl} T \colon C_p^t \otimes C_q &\to& C_q^t \otimes C_p\\
x \otimes y &\mapsto& (-1)^{pq} y \otimes x.\ea\]
This $T$ generates an action of $\Z_2$ on $C^t \otimes_{R} C$.
Also let $T$ denote the corresponding map on homomorphisms:
\[\ba{rcl} T \colon \Hom_R(C^p,C_q) &\to& \Hom_R(C^q,C_p)\\
\theta &\mapsto& (-1)^{pq} \theta^*. \ea\]
 \end{definition}

\subsection{Symmetric Poincar\'{e} complexes and closed manifolds}\label{section:symm-cxs-closed-mflds}

In this section we explain symmetric structures on chain complexes, following Ranicki~\cite{Ranicki3}.  Later on we will see that the chain complex of a manifold inherits a symmetric structure.

Take $M$ to be an $n$-dimensional closed manifold with $\pi:= \pi_1(M)$ and universal cover $\wt{M}$.
Let
$\wt{\Delta} \colon \wt{M} \to \wt{M} \times \wt{M}$; $y \mapsto (y,y)$
be the diagonal map on the universal cover of $M$.  This map is $\pi$-equivariant, so we can take the quotient by the action of $\pi$ to obtain
\begin{equation}\label{topoldiagonal2}
\Delta \colon M \to \wt{M} \times_{\pi} \wt{M},
\end{equation}
where $\wt{M} \times_{\pi} \wt{M} := \wt{M} \times \wt{M}/(\{(x,y) \sim (gx,gy) \,|\,g \in \pi)$.
The notion of a symmetric structure arises from an algebraic version of this map, as we now proceed to describe.

The Eilenberg-Zilber theorem \cite[Chapter~VI,~Corollary~1.4]{Br93} says that there is a natural chain equivalence $EZ  \colon C(\wt{M} \times \wt{M}) \simeq C(\wt{M}) \otimes_{\Z} C(\wt{M})$.
By a mild abuse of notation, let
\[\wt{\Delta_*} \colon C(\wt{M}) \to C(\wt{M}) \otimes_{\Z} C(\wt{M})\]
be the composition of the map induced on chain complexes by $\wt{\Delta}$ followed by $EZ$.
Take the tensor product over $\Z[\pi]$ with $\Z$ of both the domain and codomain, to obtain:
\[\Delta_0 \colon C(M) \to C(\wt{M}) \otimes_{\Z[\pi]} C(\wt{M}).\]
The map $\Delta_0$ evaluated on the fundamental class $[M]$ and composed with the slant map (Definition~\ref{defn:slant-map})
yields \[\Phi_0:= \backslash \Delta_0([M]) \in \Hom_{\Z[\pi]}(C^{n-*}(\wt{M}),C_*(\wt{M})).\]
In the case $n=3$ we have that $\Phi_0$ is a collection of $\Z[\pi]$-module homomorphisms of the form:
\[\xymatrix @C+1cm@R0.65cm{
C^0 \ar[r]^{\partial^*_1} \ar[d]^(0.45){\Phi_0}& C^1 \ar[r]^{\partial^*_2} \ar[d]^{\Phi_0} & C^2 \ar[d]^(0.45){\Phi_0} \ar[r]^{\partial^*_3} & C^3 \ar[d]^(0.45){\Phi_0}\\
C_3 \ar[r]_{\partial_3} & C_2 \ar[r]_{\partial_2} & C_1 \ar[r]_{\partial_1} & C_0}\]
A symmetric structure also consists of higher chain homotopies $\Phi_s \colon C^r \to C_{n-r+s}$ which measure the failure of $\Phi_{s-1}$ to be symmetric on the chain level. We will introduce the higher symmetric structures next, using the higher diagonal approximation maps.

\begin{definition}\label{Defn:higherdiagonalmaps}
A \emph{chain diagonal approximation} is a chain map $\Delta_0 \colon C_* \to C_* \otimes C_*$, with a collection, for $i \geq 1$, of chain homotopies $\Delta_i \colon C_* \to C_* \otimes C_*$ between $\Delta_{i-1}$ and $T\Delta_{i-1}$.  That is, the $\Delta_i$ satisfy the relations:
\[\partial \Delta_i - (-1)^i\Delta_i\partial = \Delta_{i-1} + (-1)^iT\Delta_{i-1}. \]
\end{definition}
The higher $\Delta_i$ give rise to the entire symmetric structure on a chain complex, as in the next definition.

\begin{definition}[Symmetric Poincar\'{e} chain complex]\label{Defn:Qgroups}
Given a finitely generated projective chain complex $C_*$ over a ring $R$, let $\Phi$ be a collection of $R$-module homomorphisms
\[\{\Phi_s \in \Hom_R(C^{n-r+s},C_r) \mid r \in \Z, s \geq 0\}\]
such that:
\[d_C\Phi_s + (-1)^r \Phi_s\delta_C + (-1)^{n+s-1}(\Phi_{s-1}+(-1)^sT\Phi_{s-1}) = 0 \colon C^{n-r+s-1} \to C_r\]
where $\Phi_{-1} = 0$.
Then $\Phi$, up to an appropriate notion of equivalence (see \cite[Part~I, section~1]{Ranicki3} for details) is called an $n$-dimensional symmetric structure.  We call $(C_*, \Phi)$ an $n$-dimensional symmetric \emph{Poincar\'{e}} complex if the maps $\Phi_0 \colon C^{n-r} \to C_r$ form a chain equivalence.  In particular this implies that they induce isomorphisms (the cap products) on homology:
\[[\Phi_0] \colon H^{n-r}(C) \xrightarrow{\simeq} H_r(C).\]
\end{definition}

The symmetric construction, which is the process by which a manifold gives rise to a symmetric chain complex, as in the next proposition, appears in \cite[Part~II,~Proposition~2.1]{Ranicki3}.

\begin{proposition}
  A closed oriented $n$-dimensional manifold $M$ gives rise to a symmetric Poincar\'{e} chain complex
 \[\big(C := C_*\big(\wt{M}\big), \Phi_i := \backslash\Delta_i([M])\big),\]
unique up to chain homotopy equivalence.
\end{proposition}

\section{The symmetric Poincar\'e chain complex of zero-framed surgery on a knot}\label{section:example-knot-exteriors}

In this section, given a diagram of an oriented knot $K$ in $S^3$, we give an algorithm to construct an explicit symmetric Poincar\'{e} chain complex for the zero-framed surgery manifold $M_K$, with coefficients in $\Z[\pi]$, where $\pi = \pi_1(M_K)$.

The organisation of this section is as follows.
In Section~\ref{section:handle-decomposition-knot-exterior} we describe a handle decomposition for the zero surgery on a knot.
In Section~\ref{section:chaincomplex} we use this to explicitly describe a cellular chain complex for the universal cover of the zero surgery on the knot, that is the cellular chain complex $C_*(M_K;\Z[\pi])$.  The only part of these sections which is not well-known is the description of the boundary maps corresponding to the attaching of 3-handles, though see \cite{Igusa-Orr:2001} for similar arguments.  Nevertheless all this material is crucial for the description and justification of the formulae for the symmetric structures in Section \ref{section:trotters-formulae}, as well as necessary fixing of the notation.

Much of the material of this section is a retract of material from the PhD thesis of the second author~\cite{Powellthesis}.
With some work, the same construction could enable us to start instead with a non-split diagram of a link, and one can easily modify the construction to use any integral surgery coefficient instead of zero.  Thus this procedure can be generalised to give the symmetric chain complex for any $3$-manifold.

\subsection{A handle decomposition from the Wirtinger presentation}\label{section:handle-decomposition-knot-exterior}

A knot diagram determines a graph in $S^2$ by forgetting crossing information.
A knot diagram is \emph{reduced} if there does not exist a region in the associated graph which abuts itself at a vertex.

\begin{definition}\label{Defn:quaddecomp}
A reduced knot diagram with a nonzero number of crossings determines a \emph{quadrilateral decomposition} of $S^2$, an expression of $S^2$ as a union of 4-sided polyhedra.  This is equivalent to a graph in $S^2$ whose complementary regions each have four edges.  A reduced knot diagram determines such a graph as the dual graph to the graph defined by the knot diagram.
  Put a vertex in each region of the graph determined by the knot, and then join a pair of the new vertices with an edge if the original regions were separated by an edge in the knot diagram graph.
\begin{figure}[h!]
 \begin{center}
 \hspace{-2cm}\includegraphics [width=7cm] {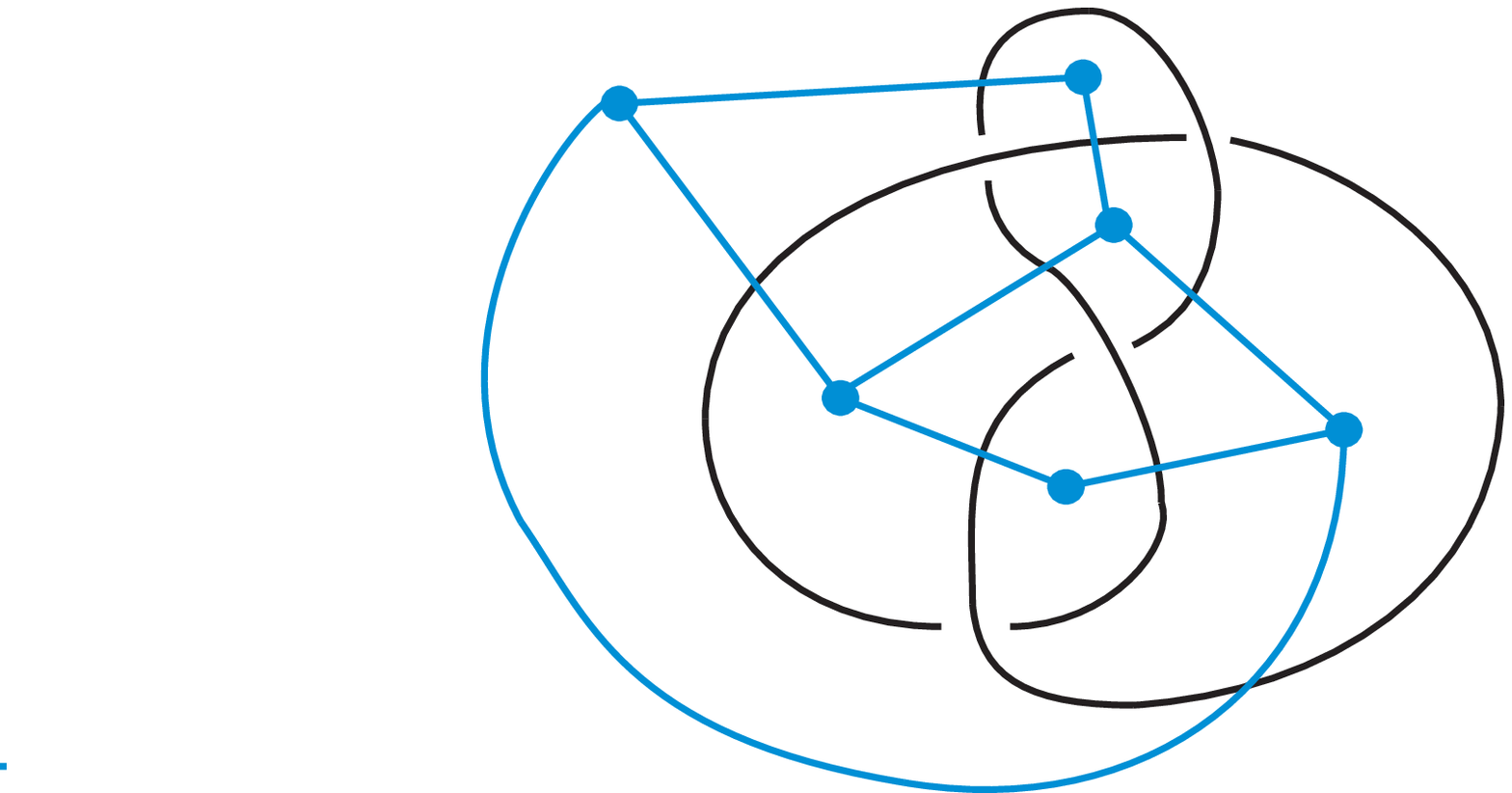}
 \caption {Quadrilateral decomposition for a diagram of the figure eight.}
 \label{Fig:quaddecomp}
 \end{center}
\end{figure}
  Each region complementary to the dual graph then has a single crossing in its interior, and since the original graph is four-valent, each region is a quadrilateral.
\end{definition}

For a knot $K \subset S^3$, we denote the \emph{knot exterior} $S^3 \sm \nu K$ by $X_K$ and the zero-surgery of $S^3$ along $K$ by $M_K$. Next, we show that one can construct a handle decomposition for the zero surgery using a number of handles proportional to the crossing number. The proof uses an explicit construction, which will enable us later to algorithmically produce the symmetric chain complex.

\begin{construction}\label{Thm:includingboundary}
Given a reduced diagram for a knot $K \subset S^3$, with $c \geq 3$ crossings, there is a handle decomposition of the zero-surgery $M_K$
 with the following handles:
\[M_K\,\, = \,\,h^0\, \cup\, \smcup{i=1}{c}\, h_i^1 \,\cup\, \smcup{j=1}{c+1}\,h_j^2 \,\cup\, \smcup{k=1}{2}\,h_k^3. \]
\end{construction}

We need to fix some conventions before we begin the construction. Choose an enumeration of the crossings in the diagram, and therefore of the regions of the quadrilateral decomposition satisfying the following condition.  For $i=1,\dots,c-1$, from crossing $i$, walk along the over-strand in the direction of the orientation.  The next over-crossing arrived at must be numbered $i+1$.  For $i=c$, the next crossing must be the crossing numbered~$1$.
 We will use the following terminology during the proof.

\begin{definition}
Let $\eps_j \in \{-1,+1\}$ be the sign of the $j$th crossing of a knot diagram.  The writhe of the diagram is $\Wr := \sum_{j=1}^c \, \eps_j.$
\end{definition}

\noindent\textit{Description of Construction~\ref{Thm:includingboundary}.}
Divide $S^3$ into an upper and lower hemisphere: $S^3 \cong D_-^3 \cup_{S^2} D_+^3$.  Let the knot diagram be in $S^2$, and arrange the knot itself to be close to its image in the diagram in $S^2$ but all contained in $D_+^3$.  Let $D_-^3$ be $h^0$, the 0-handle.  Attach $1$-handles which start and end at the $0$-handle and go over the knot, one for each edge of the quadrilateral decomposition of $S^2$.  The feet of each $1$-handle should  be contained in small discs around the vertices of the quadrilateral decomposition.

There are $c$ regions and therefore $2c$ edges and currently $2c$ 1-handles.  Now, for each crossing, attach a 2-handle which goes between the strands of the knot, so that the 1-handles that go over the under crossing strand and this 2-handle can be amalgamated into a single 1-handle by handle cancellation. There are now $c$ 1-handles.  Enumerate the 1-handles, labelling them $h^1_1,\dots,h^1_c$.  Figure~\ref{Fig:bigcrossing3} shows the final configuration on 1-handles at each crossing.  In Figure~\ref{Fig:bigcrossing3}, the 1-handles associated to crossing $i$ are labelled $h^1_{i_1}$, $h^1_{i_2}$ and $h^1_{i_3}$.  This defines, for each $i=1,\dots,c$, three numbers $i_1$, $i_2$ and $i_3$.  We choose our enumeration so that $1_1=1$.

\begin{figure}[h]
  \begin{center}
    {
    \psfrag{C}{$h^1_{i_1}$}
    \psfrag{D}{$h^1_{i_2}$}
    \psfrag{B}{$h^1_{i_3}$}
    \includegraphics[width=9cm]{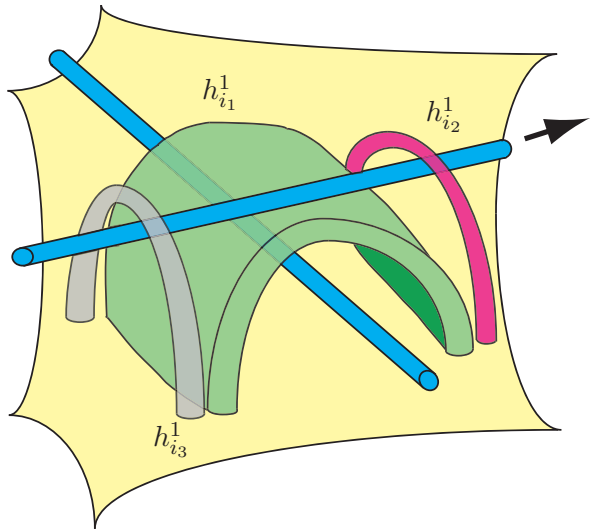}}
  \caption{The 1-handles.}
  \label{Fig:bigcrossing3}
  \end{center}
\end{figure}

The next step is to attach 2-handles.  For each crossing, and therefore region of the quadrilateral decomposition of $S^2$, we glue a 2-handle on top of the knot, with boundary circle going around the 1-handles according to the boundary of the region of $S^2$, as shown in Figure \ref{Fig:bigcrossingfinish}.

\begin{figure}[h]
  \begin{center}
    \includegraphics[width=10cm]{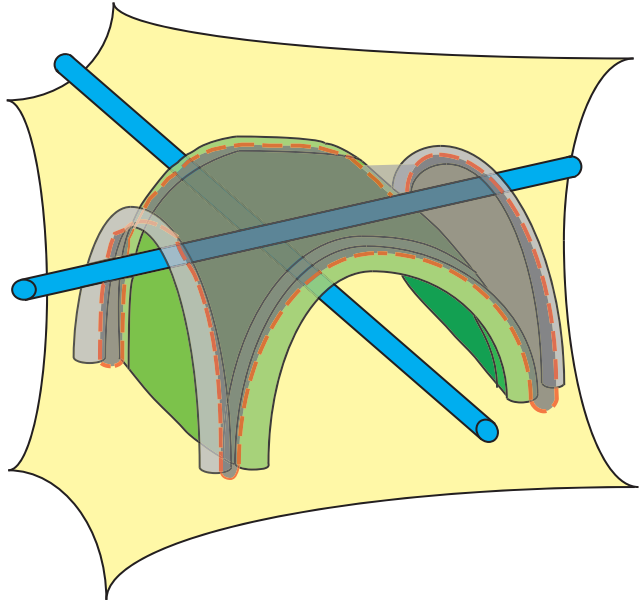}
  \caption{The 2-handle attachment.}
  \label{Fig:bigcrossingfinish}
  \end{center}
\end{figure}

Finally, after a 2-handle is attached over each crossing of the knot, we have $c$ 2-handles, and the upper boundary of the 2-skeleton is again homeomorphic to $S^2$.  This means that we can attach a 3-handle $h^3_1$ to fill in the rest of $S^3$. This completes the description of a handle decomposition of $X_K$.
We pause briefly to observe that the handle decomposition of $X_K$ constructed so far, in particular the attaching maps of the 2-handles, can be used to read off the well known Wirtinger presentation of the fundamental group.

\begin{proposition}[Wirtinger presentation]\label{prop:wirtingerpresentation}
Define relators
\[r_i :=
\begin{cases}
g_{i_2}^{-1}g_{i_1}^{-1}g_{i_3}g_{i_1} & \text{if crossing }i \text{ is of sign }+1;\\
g_{i_2}^{-1}g_{i_1}g_{i_3}g_{i_1}^{-1} & \text{if crossing }i \text{ is of sign }-1.
\end{cases}\]
The $r_i$ give rise to a presentation for the fundamental group of $X_K$
\[\pi_1(X_K) \cong \langle \; g_1,\dots,g_{c}\,|\, r_1,\dots,r_c\,\rangle.\]
\end{proposition}
One of the relators could be cancelled, but we do not want to do this, for reasons related to constructing the symmetric structure later.

Now we complete our decomposition of $M_K$ by attaching one more 2-handle, $h^2_{c+1}= h^2_s$, and another 3-handle, $h^3_2= h^3_s$, to the boundary of $X_K$.
To attach the final 2-handle, we need to see how the longitude lives in our handle decomposition.  Look again at Figure \ref{Fig:bigcrossingfinish}, and imagine the longitude as a curve following the knot, just underneath it.
Since the writhe of the diagram is potentially non-zero, in order to have the zero-framed longitude, we take it to wind $(-\Wr)$ times around the knot, underneath the tunnel created by the 1-handle $h^1_1$.  Then deform the longitude towards the 0-handle $h^0_o$, everywhere apart from underneath $h^1_1$. We see that at the over-strand of crossing $j$, the longitude follows the 1-handle $h^1_{j_1}$, respecting the orientation if $\eps_j = 1$ and opposite to the orientation if $\eps_j=-1$.  As it follows under-strands, we deform it to the 0-handle, so these have no contribution to the longitude as a fundamental group element.  However, within the tunnel underneath $h^1_1$, we act differently, and instead deform the longitude outwards to see that it follows $h^1_1$, $(-\Wr)$ times.  A word for the longitude, as an element of $\pi_1(X_K)$, in terms of the Wirtinger generators, is
\begin{align}\label{eqn:longitudeword}
\ell = g_1^{-\Wr}g_{k_1}^{\eps_k}g_{(k+1)_1}^{\eps_{k+1}}\dots g_{(k+c-1)_1}^{\eps_{k+c-1}}.\end{align}
Here $k$ is the number of the crossing reached first as an over crossing, when starting on the under crossing strand of the knot which lies in region 1; the indices $k, k+1,\dots,k+c-1$ are to be taken mod $c$, with the exception that we prefer the notation $c$ for the equivalence class of $0 \in \Z_c$. Finally, after attaching this 2-handle to $\partial X_K= T^2$, we have a boundary $S^2$ that can be capped off with a 3-handle $h^3_2= h^3_s$ to fill in the rest of $M_K$.  This completes Construction~\ref{Thm:includingboundary}.\qed

\subsection{A cellular chain complex of the universal cover of a knot exterior}\label{section:chaincomplex}

In this section we use our handle decomposition of a zero-surgery $M_K$ from the previous section to write down a cellular chain complex $C_*(M_K;\Z[\pi])$ of the universal cover, where $\pi:= \pi_1(M_K)$.  Note that a handle decomposition gives rise to a CW structure on a homotopy equivalent space, and by a slight abuse of notation we refer to the cells by the same symbols.  We could have worked with a cell decomposition from the outset, but we find handles easier to visualise and therefore find it easier to verify that we obtain a chain complex for the correct space.  We only ever work with symmetric chain complexes up to chain equivalence, so nothing is lost by passing to a homotopy equivalent space.

We will use the free differential calculus and the notion of an identity of the presentation to give the formulae for the boundary maps.
An element of $\pi_1(M_K)$ is represented by a word $w$ in $F$, the free group on $g_1,\dots,g_c$.  This in turn determines a path in the 1-skeleton of the universal cover $\wt{M_K}^{(1)}$, which in the case $w=r_i$ is a lift of the attaching circle of a 2-handle $h^2_i$.  The free differential calculus, due to Fox \cite{Fox2}, is a formalism that tells us which chain this path is in $C_1(M_K;\Z[\pi])$.

\begin{definition}\label{freederivative}
The \emph{free derivative} with respect to a generator $g_i$ of a free group $F$ is a map $\frac{\partial}{\partial g_i} \colon F \to \Z[F]$ defined inductively, using the following rules:
\[\lmfrac{\partial (1)}{\partial g_j} = 0; \;\;\; \lmfrac{\partial g_i}{\partial g_j} = \delta_{ij}; \;\;\; \lmfrac{\partial(uv)}{\partial g_j} = \lmfrac{\partial u}{\partial g_j} + u \lmfrac{\partial v}{\partial g_j}.\]
Extend this linearly to make the free derivative into a function $\Z[F] \to \Z[F]$.
\end{definition}

\begin{definition}\label{identitypresentation}
Let $G$ be a group with presentation $\langle g_1, \dots, g_c | r_1, \dots r_d \rangle.$
Let $F$ be the free group with generators $g_1, \dots, g_c$.
Let $P$ be the free group on letters $\rho_1,\dots,\rho_{d}$, and let $\psi\colon P \ast F \to F$ be the homomorphism such that $\psi(\rho_i)=r_i$ and $\psi(g_j)=g_j$.
 An \emph{identity of the presentation} is a word in $\ker(\psi) \leq P \ast F$ that can be written as a product of words of the form $w \rho_j^{\eps} w^{-1}$, where $w \in F$, $j \in \{1, \dots, d\}$, and $\eps \in\{\pm1\}$.
\end{definition}

\begin{conventions}\label{conventions2}
Our chain groups are based free left $\Z[\pi]$-modules.  We denote the module freely generated by a basis $e_1,\dots,e_m$ by $\ll e_1,\dots,e_m \rr$.
We define module homomorphisms only on the basis elements of a free module, and use the left $\zp$ module structure to define the map on the whole module.  This has the effect, in the non-commutative setting, that when we want to formally represent elements of our based free modules as vectors with entries in $\Z[\pi]$ detailing the coefficients, then the vectors are written as row vectors, and the matrices representing a map must be multiplied on the right.  This is because the order of multiplication of two matrices should be preserved when multiplying elements to calculate the coefficients.
\end{conventions}

The handle decomposition from Section~\ref{section:handle-decomposition-knot-exterior} gives rise to the chain complex described in detail in the next theorem.
The set up is as follows.
Let $K$ be a knot with zero-surgery $M_K$ and suppose we have a reduced knot diagram for $K$ with $c \geq 3$ crossings.  Denote the free group on the letters $g_1,\dots,g_c$ by $F=F(g_1, \dots ,g_c)$, and let $\ell \in F$ be the word for the longitude defined in Construction~\ref{Thm:includingboundary}. Recall that
\[\ell = g_1^{-\Wr}g_{k_1}^{\eps_k}g_{(k+1)_1}^{\eps_{k+1}}\dots g_{(k+c-1)_1}^{\eps_{k+c-1}},\]
where $k$ is the number of the crossing reached first as an over crossing, when starting on the under crossing strand of the knot which lies in region 1; the indices $k, k+1,\dots,k+c-1$ are to be taken mod $c$, using the representative $c$ for the equivalence class of $0 \in \Z_c$.  The sign of crossing $j$ is $\eps_j$ and $\Wr$ is the writhe of the diagram, which is the sum of the $\eps_j$.

There is a presentation $\pi = \pi_1(M_K) = \langle\,g_1,\dots,g_c, \,|\,r_1,\dots,r_c, r_s=\ell\,\rangle$
with the Wirtinger relators $r_1,\dots,r_c \in F(g_1,\dots,g_c)$ read off from the knot diagram, and $\ell$ as above.

\begin{theorem}\label{Thm:mainchaincomplex}
The cellular chain complex $C_*(M_K;\Z[\pi])$ corresponding to the handle decomposition of Construction~\ref{Thm:includingboundary} is given below.  The correspondence of the free module basis elements to the handles from Construction~\ref{Thm:includingboundary} is also given.  More precisely, each handle corresponds to a cell in a CW complex that is homotopy equivalent to $M_K$, and the cells correspond to basis elements.  Recall from Conventions~\ref{conventions2} that matrices multiply on row vectors on the right.
\[\underset{\cong \langle h^3_1,h^3_{s} \rangle}{\underbrace{ \bigoplus_{2}{}\,\Z[\pi] }}\,\,
\xrightarrow{\partial_3}\,\,
\underset{\cong \langle h^2_1,\dots,h^2_c,h^2_s \rangle }{\underbrace{\bigoplus_{c+1}{}\,\Z[\pi] }}\,\, \xrightarrow{\partial_2}\,\,
\underset{\cong \langle h^1_1,\dots,h^1_c \rangle }{\underbrace{\bigoplus_{c}{}\,\Z[\pi]  }}\,\, \xrightarrow{\partial_1}\,\,
\underset{\cong \langle h^0 \rangle}{\underbrace{ \Z[\pi] }}
\]
where:
\[\ba{rcl} \partial_3 &=& \left(
               \begin{array}{cccccc}
                 \hfill w_1 &  & \hdots &  &\hfill w_c & 0  \\
                 -u_1 &  & \hdots &  & -u_c & g_1-1 \\
               \end{array}
             \right);\\ & & \\
\partial_2 &=& \left(
                 \begin{array}{ccccc}
                   \left(\partial r_1/\partial g_1\right) &  & \hdots &  & \left(\partial r_1/\partial g_c\right)\\
                    &  &  &  &  \\
                   \vdots &  & \ddots &  & \vdots \\
                    &  &  &  &    \\
                   \left(\partial r_c/\partial g_1\right) &  & \hdots &  & \left(\partial r_c/\partial g_c\right) \\
                   \left(\partial \ell/\partial g_1\right) &  & \hdots &  & \left(\partial \ell/\partial g_c\right)\\
                 \end{array}
               \right);\text{ and} \\ & & \\
\partial_1 &=& \left(
               \begin{array}{ccccccc}
                 g_1 - 1 &  & \hdots &  & g_c - 1\\
               \end{array}
             \right)^T. \ea\]

The words $u_{k+i}$ in $\partial_3$ are given by $g_1^{1-\Wr}$ followed by the next $i$ letters in the word for the longitude:
\[u_{k+i} = g_1^{1-\Wr}g_{k_1}^{\eps_k}g_{(k+1)_{1}}^{\eps_{k+1}}\dots g_{(k+i)_1}^{\eps_{k+i}}.\]

To determine the words $w_i$ arising in $\partial_3(h^3_1)$, consider the quadrilateral decomposition of the knot diagram (Definition \ref{Defn:quaddecomp}).  At each crossing $i$, we have an  edge that we always list first in the relation, $g_{i_2}$.  Choose the vertex, call it $v_i$, which is at the end of $g_{i_2}$ (corresponding to the handle $h^1_{i_2}$ in Figure~\ref{Fig:bigcrossing3}).  For crossing $i$, choose a path in the 1-skeleton of the quadrilateral decomposition from $v_1$ to $v_i$.  This yields a word $w_i$ in $g_1,\dots,g_c$.  Then the component of $\partial_3(h^3_1)$ along $h^2_i$ is $w_i$.
\end{theorem}

\begin{proof}
  There exists a choice of basing for the handles given in Construction~\ref{Thm:includingboundary} such that the boundary maps in the associated cellular chain complex are as given.  We refer to \cite{Powellthesis} for full details.  Here, we only elaborate on the boundary maps of the 3-cells.

Note that there is a correspondence between 3-cells and identities of a presentation.
The boundary of a 3-cell is a 2-sphere, which is attached to a collection of 2-cells.  Remove one copy of one of these 2-cells from the boundary. The remainder is a disc that consists of a union of 2-cells.  This union of 2-cells together with a choice of basing path determines a product of conjugates of relators.  Adding the inverse of the relator corresponding to the 2-cell that we removed yields an identity of the presentation.  Conversely, an identity of a presentation gives rise to a map of $S^2$ into the presentation 2-complex, to which one can attach a 3-cell.

In particular, the 3-cells of $M_K$, $h^3_1$ and $h^3_{2}$, correspond to the following identities of the presentation of $\pi_1(M_K)$, whence the $\partial_3$ map above is derived:
\begin{equation}\label{equation:identity-s-o}
  s_1 = \smprod{k=1}{c} \, w_{j_k}r_{j_k}w_{j_k}^{-1} = 1;
\end{equation}
\begin{equation}\label{equation:identity-s-partial}
 s_{2} = \bigg(\,\tmprod{j=0}{c-1}\,u_{k+j}r^{-1}_{k+j}u_{k+j}^{-1}\bigg) \cdot \big( g_1 r_s g_1^{-1} \big) \cdot r_s^{-1} = 1.
\end{equation}
The set $\{j_k \mid k =1,\dots,c\}$ is equal to $\{1,\dots,c\}$, but we will not be concerned with the precise order.
\end{proof}

\begin{remark}\label{Rmk:relativefundclass}
Passing to $\Z$ coefficients, the 3-dimensional chain
\[[M_K] := -h^3_1 - h^3_2 \in C_3(M_K;\Z) = \Z \otimes_{\Z[\pi]} C_3(M_K;\Z[\pi])\]
represents a cycle in $C_3(M_K;\Z)$.
 This is our choice of \emph{fundamental class} for the zero surgery. We shall use the image of this class under a diagonal chain approximation map in Section \ref{section:trotters-formulae} to derive the symmetric structure on the chain complex.
\end{remark}

Note that the  chain complex that we have constructed is algorithmically extractable from a knot diagram.

\subsection{Formulae for the diagonal chain approximation map}\label{section:trotters-formulae}

Trotter \cite{Trotter} gave explicit formulae, which we shall now exhibit, for a choice of diagonal chain approximation map on the 3-skeleton of a $K(\pi,1)$, given a presentation of $\pi$ with a full set of identities for the presentation. (A set of identities for a presentation is called \emph{full} if  the corresponding CW complex has trivial second homotopy group or, equivalently, if any other identity is a product of conjugates of identities in this set.)
Note that for any nontrivial knot $K$, the zero-surgery $M_K$ is an Eilenberg-MacLane space $K(\pi_1(M_K),1)$, by work of Gabai~\cite[Corollary~5]{Gabai:1986-2}.

Let $\pi$ be a group with a presentation $\langle \,g_1,\dots,g_a \mid r_1,\dots,r_b\,\rangle$ with a full set of identities
 \[\big\{s_m = \prod_{k=1}^{d_m} \, w_{j_k^m} r_{j_k^m}^{\eps_{j_k^m}} (w_{j_k^m})^{-1}\big\}_{m=1}^e\] for the presentation.
 Let $Y$ be a model for $K(\pi,1)$, and suppose that $Y$ has a CW structure that corresponds to the presentation and identities. Let $\wt{Y}$ be the universal cover of $Y$.
Let $\{h^j_i\}$ be the basis elements of the $\Z[\pi]$-modules $C_j(\wt{Y})$, for $(0 \leq j \leq 3)$, where each $h_i^1$ correspond to a $j$-cell of $Y$, and the $1$-cells correspond to the $a$ generators of $\pi$, $2$-cells correspond to the $b$ relations, and $3$-cells correspond to the $e$ identities.  Let $\a \colon F(g_1,\dots,g_a) \to C_1(\wt{Y})$ be given by $\alpha(v) = \sum_i \, \frac{\partial v}{\partial g_i} h^1_i$, using the Fox derivative (Definition \ref{freederivative}). Let $\g \colon F \to C_1(\wt{Y}) \otimes C_1(\wt{Y})$ be a map with $\g(1)=\g(g_i)=0$ that satisfies
\begin{equation}\label{gammadef}
\g(uv)=\g(u) +u\g(v) + \a(u) \otimes u\a(v).
\end{equation}
Note that in fact $\g$ is entirely determined by the choice of $\g(g_i)$ and Equation~\ref{gammadef}, as proven by Trotter in~\cite[page~472]{Trotter}.
In particular note that with $u=g_i$ and $v=g_i^{-1}$, Equation~\ref{gammadef} implies that $\g(g_i^{-1})= g_i^{-1} h^1_i \otimes g_i^{-1} h^1_i$.

\begin{example}
We give the result of the calculation of $\g$ for a typical word which arises in the Wirtinger presentation of the knot group:
\begin{eqnarray*}
\g(g_i^{-1}g_kg_jg_k^{-1}) & = & (g_i^{-1}h^1_i \otimes g_i^{-1}h^1_i) - (g_i^{-1}h^1_i \otimes g_i^{-1}h^1_k) + (h^1_k \otimes h^1_k) -\\ & & (g_i^{-1} g_k h^1_j \otimes h^1_k) + (g_i^{-1} h^1_k - g_i^{-1}h^1_i) \otimes (g_i^{-1}g_k h^1_j - h^1_k).
\end{eqnarray*}
\end{example}

\begin{theorem}\label{trotterdelta0}\cite{Trotter}
  A chain diagonal approximation map $\Delta_0 \colon C(\wt{Y}) \to C(\wt{Y}) \otimes_{\Z} C(\wt{Y})$ can be defined on the $3$-skeleton $\wt{Y}^{(3)}$ as follows.
\begin{eqnarray*}
\Delta_0(h^0) & = & h^0 \otimes h^0 \\
\Delta_0(h^1_i) & = & h^0 \otimes h^1_i + h^1_i \otimes g_i h^0 \\
\Delta_0(h^2_j) & = & h^0 \otimes h^2_j + h^2_j \otimes h^0 - \g(r_j) \\
\Delta_0(h_m^3) & = & h^0 \otimes h^3_m + h^3_m \otimes h^0 + \tmsum{k=1}{d_m} \, \eps_{j_k^m} \left(\a(w_{j_k^m}) \otimes w_{j_k^m} h^2_{j_k^m} + w_{j_k^m} h^2_{j_k^m} \otimes \a(w_{j_k^m})\right)\\
& & + \tmsum{k=1}{d_m} \delta_{j_k^m} w_{j_k^m} (h^2_{j_k^m} \otimes \alpha(r_{j_k^m})) - \tmsum{1\, \leq\, \ell \, <\, k\, \leq \, d_m}{} \eps_{j_{\ell}^m} w_{j_{\ell}^m} h^2_{j_{\ell}^m} \otimes \eps_{j_k^m} w_{j_k^m} \a(r_{j_k^m})
\end{eqnarray*}
where $\delta_{i} = \frac{1}{2}(\eps_{i} - 1)$.
\end{theorem}

\begin{proof}
See \cite[pages 475--6]{Trotter}, where Trotter shows that these are indeed chain maps i.e.\ that $\Delta_0 \circ \partial = d \circ \Delta_0.$
Trotter does not state his sign conventions explicitly; however, careful inspection of his calculations shows that his convention for the boundary map of $C^t \otimes C$ disagrees with ours.  We therefore undertook to rework his proof using our sign convention.  It turned out that the only change required in the formulae was a minus sign in front of $\g(r_j)$, which alteration we have already made for the statement of the theorem.
\end{proof}

We do not have explicit formulae for the higher diagonal maps $\Delta_i$, for $i=1,2,3$, but a result of J.~Davis \cite{Davis} ensures that they exist for a $K(\pi,1)$.
For our purposes we only require explicit knowledge of $\Delta_0$.

Now we specialise to the case $Y=M_K$, which as noted above is a 3-dimensional Eilenberg-Maclane space $K(\pi_1(M_K),1)$ for $K \neq U$.  Recall from Section~\ref{section:basic-chain-cx-constructions} that we obtain symmetric structure maps as \[\Phi_0 = \backslash \Delta_0([M_K]) \in \Hom_{\Z[\pi]}(C^{3-*}(M_K;\Z[\pi]),C_{*}(M_K;\Z[\pi])).\]
In practice, in our application to twisted Blanchfield pairings below, we shall use the map $\Phi_0^*$ instead of $\Phi_0$, since this is a much simpler map in Trotter's formulae of Theorem~\ref{trotterdelta0}. Therefore next we describe the map $\Phi_0^* \colon C^2 \to C_1$ explicitly.
From Trotter's formulae, we have that an identity of the presentation $\prod_k \, w_{k} r_{k}^{\eps_{k}} w_{k}^{-1}$ corresponding to a $3$-handle $h^3$ gives rise to a term in $\Delta_0([M_K])$ of the form
\[\smsum{k}{} \Big(\smsum{i}{} \eps_k \smfrac{\partial w_k}{\partial g_i} h^1_i \otimes w_k h^2_k \Big).\]
We obtain the following matrix over $\Z [\pi]$.
\[\Phi_0^*= \ol{\Phi_0}^T= \left(\begin{array}{ccccc}
 -w_1^{-1}\frac{\partial w_1}{\partial g_1} + u_1^{-1}\frac{\partial u_1}{\partial g_1} & & \hdots & & -w_1^{-1}\frac{\partial w_1}{\partial g_c} + u_1^{-1}\frac{\partial u_1}{\partial g_c} \\
 & & & &  \\
 \vdots & & \ddots & & \vdots  \\
 & & & &  \\
 -w_c^{-1}\frac{\partial w_c}{\partial g_1} + u_c^{-1}\frac{\partial u_c}{\partial g_1} & & \hdots & & -w_c^{-1}\frac{\partial w_c}{\partial g_c} + u_c^{-1}\frac{\partial u_c}{\partial g_c} \\
 g_1^{-1} & & \hdots & & 0
 \end{array} \right)\]

Hopefully the following simple example illustrating our conventions will be useful to the reader.

\begin{example}
Consider the diagram of the left handed trefoil, with arcs labelled $g_1, g_2,$ and $g_3$ as on the left of Figure~\ref{Fig:trefoilexample}. Label the crossings as in Construction~\ref{Thm:includingboundary}.
  Note that the requirement that the crossing labelled $1$ must have $g_1$ as its under crossing strand determines the rest of the crossing labels.

\begin{figure}[h]
  \begin{center}
  \includegraphics[height=3cm]{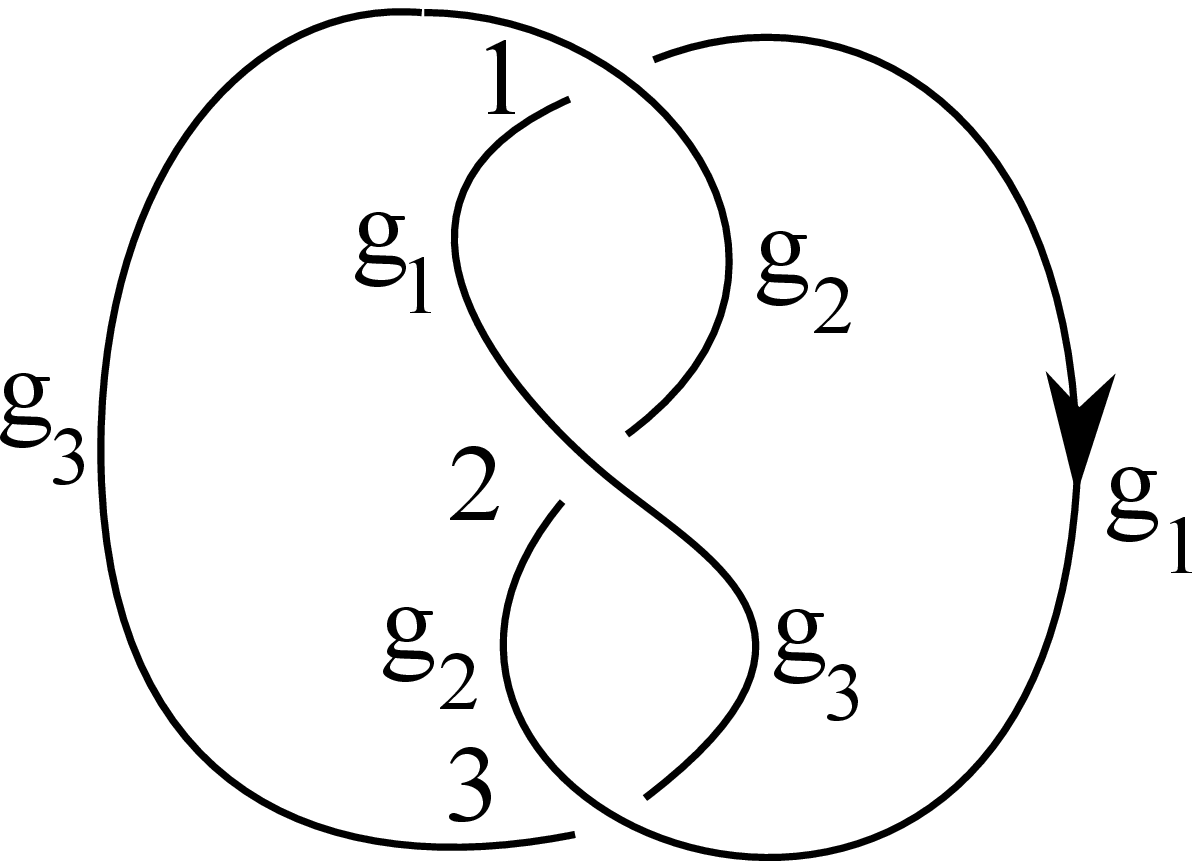}
  \qquad \qquad
   \includegraphics[height=4cm]{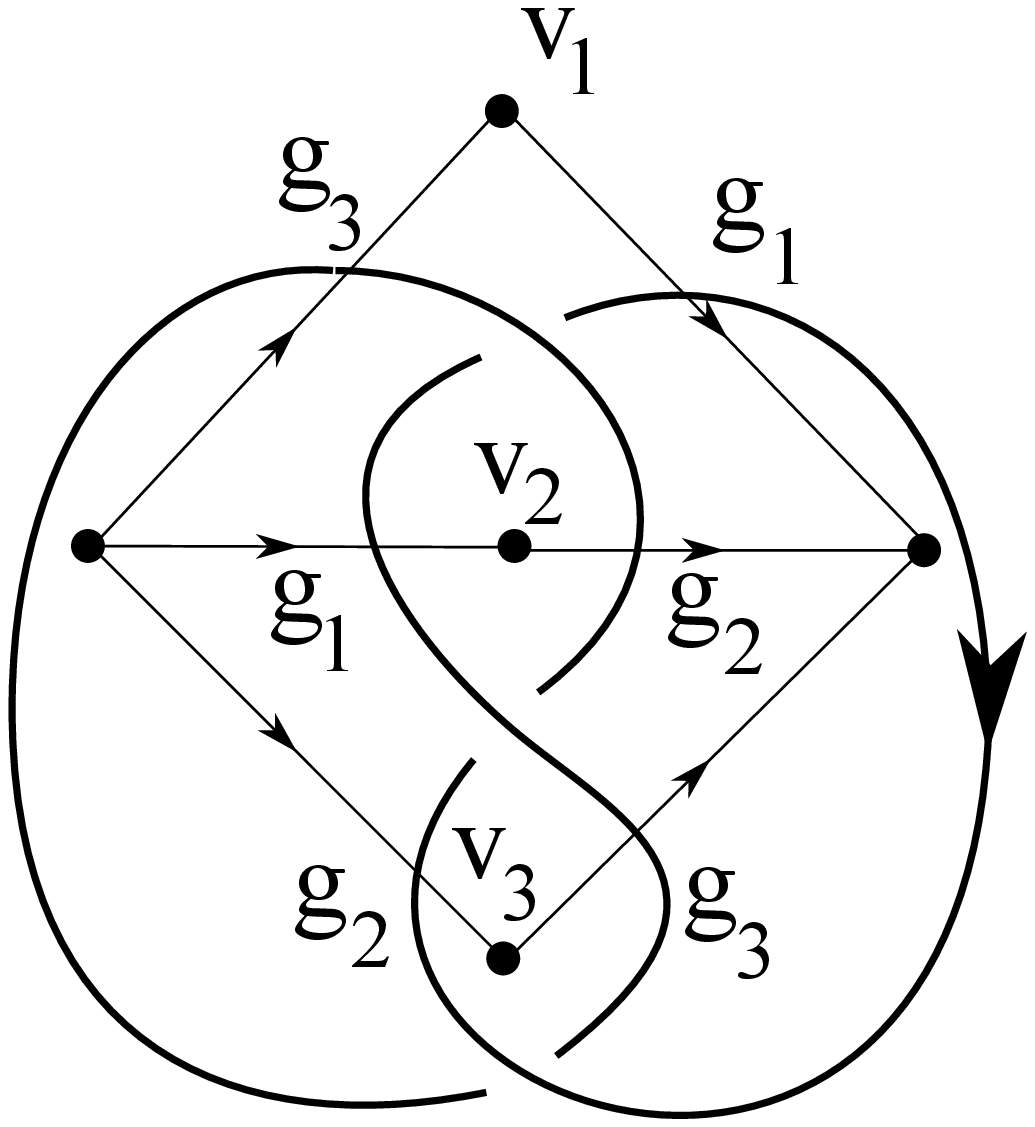}
  \caption{Computing $C_*(M_K, \Z[\pi_1(M_K)])$ for $K$ the left handed trefoil.}
  \label{Fig:trefoilexample}
  \end{center}
\end{figure}

We therefore have Wirtinger relations $r_1= g_3^{-1}g_1 g_2 g_1^{-1}$, $r_2= g_1^{-1} g_2 g_3 g_2^{-1}$, and $r_3= g_2^{-1} g_3 g_1 g_3^{-1}$,  as in Proposition~\ref{prop:wirtingerpresentation}. Since our diagram has $\Wr=-3$, we have that the zero-framed longitude is $\ell= g_1^3 g_3^{-1} g_1^{-1} g_2^{-1}$. Follow the formulae of Theorem~\ref{Thm:mainchaincomplex} to obtain words $u_1=g_1^4 g_3^{-1}$, $u_2= g_1^4 g_3^{-1} g_1^{-1} $, and $u_3=g_1^4 g_3^{-1}g_1^{-1} g_2^{-1}$.
Finally, we need to compute the words $w_i$. Consider the right side of Figure~\ref{Fig:trefoilexample}, which gives the quadrilateral decomposition corresponding to our diagram. Following the instructions of Theorem~\ref{Thm:mainchaincomplex}, we label the vertices that correspond to each crossing as shown.
 Note that in this example the vertices associated to crossings are mutually distinct, though this is not generally the case. Choose arbitrary paths from $v_1$ to $v_i$ to yield the words $w_1=1$, $w_2=g_3^{-1} g_1$, and $w_3= g_3^{-1} g_2$.
It is now a straightforward exercise in Fox calculus to write down explicit matrix representatives for $\partial_3, \partial_2, \partial_1$, and $\Phi_0^*$ according to the formulae of Theorem~\ref{Thm:mainchaincomplex} and the above equations.
\end{example}

\section{Definition of the twisted Blanchfield pairing}\label{section:defn-TBF}

\subsection{Twisted homology and cohomology groups}\label{section:twistedhomology}
Let $Y$ be a connected topological space with a base point $y_0$.  Write $\pi=\pi_1(Y,y_0)$ and denote  the universal cover of $Y$ by $p\colon \wt{Y}\to Y$. Let $Z\subset Y$ be a subspace of $Y$ and write $\wt{Z}=p^{-1}(Z)$.
The group $\pi$ acts on $C_*(\wt{Y})$ on the left. Thus we view $C_*(\wt{Y},\wt{Z})$ as a chain complex of free left $\Z[\pi]$-modules.

Let $R$ be a commutative domain with involution and let $Q$ be its field of fractions.
Here an involution on a ring $R$ is an additive self map $a \mapsto \ol{a}$ with $\ol{a\cdot b} = \ol{b}\cdot \ol{a}$, $\ol{1}=1$ and $\ol{\ol{a}}=a$.
For example, given a group $G$  we will always view $\Z[G]$ as a ring equipped with the involution $\ol{\sum_{g\in \pi}n_gg}=\sum_{g\in \pi}n_gg^{-1}$.  A left $R$-module $N$ becomes a right $R$-module using the involution, via the action $n \cdot a := \ol{a}n$ for $r\in R, n \in N$.  Denote this right module by $N^t$.  A similar statement holds with left and right switched.  We use the same notation $N^t$ in both instances.
Modules will be left modules by default.

Let $N$ be a $(R,\Z[\pi])$-bimodule.
We write
\[ \ba{rcl} C_*(Y,Z;N)&=&N\otimes_{\Z[\pi]} C_*(\wt{Y},\wt{Z}),\\
C^*(Y,Z;N)&=&\hom_{\Z[\pi]}(C_*(\wt{Y},\wt{Z})^t,N).\ea\]
These are chain complexes of left $R$-modules.
We denote the corresponding homology and cohomology modules by $H_*(Y,Z;N)$ and $H^*(Y,Z;N)$, respectively. As usual we drop the $Z$ from the notation when $Z=\emptyset$.

We will make use of the following observation.
If $\varphi\colon \pi\to \Gamma$ is a homomorphism and if  $N$ is a $(R,\Z[\Gamma])$-bimodule, then we can view $N$ as an $(R,\Z[\pi])$-module via~$\varphi$.
Furthermore, if $q\colon \what{Y}\to Y$ denotes the covering corresponding to $\ker(\varphi)$, then for  $\what{Z}=q^{-1}(Z)$ the projection map $\wt{Y}\to \what{Y}$ induces canonical isomorphisms
\[ \ba{rclcl} C_*(Y,Z;N)&=&N\otimes_{\Z[\pi]} C_*(\wt{Y},\wt{Z})&\cong & N\otimes_{\Z[\Gamma]} C_*(\what{Y},\what{Z}),\\
C^*(Y,Z;N)&=&\hom_{\Z[\pi]}(C_*(\wt{Y},\wt{Z})^t,N)&\cong&\hom_{\Z[\Gamma]}(C_*(\what{Y},\what{Z})^t,N)\ea\]
of chain complexes of left $R$-modules.


\subsection{The evaluation map}\label{section:kronecker}

We continue with the notation from the previous section.
Let $N$ and $N'$ be $(R,\Z[\pi])$-bimodules. Let $S$ be an $(R,R)$-bimodule.  Furthermore let
\[\ll -,-\rr \colon N\times N'\to S\]
be a pairing with the property that
\[ \ll ng,n'\rr =\ll n,n'g^{-1}\rr \]
for all $n\in N$, $n'\in N'$ and $g\in \pi$. Assume the pairing $\ll -,-\rr$ is sesquilinear with respect to $R$, in the sense that
\[ \ll rn,sn'\rr =r\ll n,n'\rr \ol{s}\]
for all $n\in N$, $n'\in N'$ and $r,s\in R$.
The datum of such a pairing $\ll -,-\rr$ is equivalent to the datum of an
 $(R,R)$-bimodule homomorphism  $\Theta \colon N \otimes_{\Z[\pi]} (N')^t \to S$,
 where we remind the reader that $(N')^t$  means that the $R$ and the $\Z[\pi]$ module structures have both been involuted.
It is straightforward to verify that
\[ \ba{rcl}\kappa\colon \Hom_{\op{right-}\Z[\pi]}\big(C_*(\wt{Y},\wt{Z})^t,N\big)&\to& \op{Hom}_{\op{left-}R}(N'\otimes_{\Z[\pi]} C_*(\wt{Y},\wt{Z}),S)^t\\
f&\mapsto & (n'\otimes \sigma)\mapsto  \ll n', f(\sigma) \rr\\
\ea\]
is a well-defined isomorphism of chain complexes of left $R$-modules.
The isomorphism of chain complexes above induces an isomorphism
\[\kappa \colon H^i(Y,Z;N)\xrightarrow{\cong} H_i(\op{Hom}_{\op{left} R}(N' \otimes_{\Z[\pi]} C_*(\wt{Y},\wt{Z}),S)^t)\]
of left $R$-modules.
Finally we also consider the evaluation map
\[\ev \colon H_i(\op{Hom}_{\op{left} R}(N'\otimes_{\Z[\pi]} C_*(\wt{Y},\wt{Z}),S)^t)\to \op{Hom}_{\op{left} R}(H_i(Y,Z;N'),S)^t\]
of left $R$-modules.

We will use the following special case.
Given a representation $\a\colon \pi\to \op{GL}(R,d)$, let $R^d_\alpha$ be the $R$-module $R^d$  equipped with the right $\Z[\pi]$-module structure given by right multiplication by $\alpha(g)$ on row vectors. Furthermore let $\alphas$ be the representation given by $\alphas(g)=\ol{\alpha(g^{-1})^T}$. Finally let $Q$ be the quotient field of $R$.

Now we view $Q/R \otimes_{R} R^d_{\alpha}$ as a right $\Z[\pi]$-modules where $\Z[\pi]$ acts on the second term.
The map
\begin{equation}\label{equation:pairing-for-evaluation-map}
 \ba{rcl} \ll-,-\rr \colon Q/R \otimes_{R} R^d_{\alpha} \times R^d_{\alphas} & \to & Q/R\\
(p \otimes v,w) & \mapsto & v\ol{w^T}\cdot p\ea
\end{equation}
has the desired properties, with $N = Q/R \otimes_R R^d_{\alpha}$, $N' = R^d_{\alphas}$ and $S=Q/R$. In addition, if $\a\colon \pi\to U(k)$ is a unitary representation, then $\a=\alphas$.

\subsection{The Bockstein map}\label{section:bockstein}
Let $R$ be a commutative domain, let $N$ be an $(R,\Z[\pi])$-bimodule, and let $Q$ be the quotient field of $R$.  Note that $R$, $Q$ and $Q/R$ are $(R,R)$-bimodules.
View $Q \otimes_R N$ and $Q/R \otimes_{R} N$ as $(R,\Z[\pi])$-bimodules.

Now let $C_*$ be a chain complex of free right $\Z[\pi]$-modules.
There exists a short exact sequence
\[0\to   \Hom_{\Z[\pi]}(C_*,R \otimes_{R} N)
\to \Hom_{\Z[\pi]}(C_*, Q \otimes_{R} N)\to \Hom_{\Z[\pi]}(C_*,Q/R \otimes_{R} N) \to 0\]
of $ R$-modules.\
As usual we can identify $R \otimes_{R} N$ with $N$.
The short exact sequence above gives rise to a long exact sequence
\begin{align*}
 \dots  \to &H_i(\Hom_{\Z[\pi]}(C_*, R \otimes_{R} N))
\to H_i(\Hom_{\Z[\pi]}(C_*, Q \otimes_{R} N))\to \\
&\to  H_i(\Hom_{\Z[\pi]}(C_*, Q/R \otimes_{R} N)) \xrightarrow{BS}   H_{i+1}(\Hom_{\Z[\pi]}(C_*,R \otimes_{R} N))
\to \dots
\end{align*}
The coboundary map $BS$ in this long exact sequence is called the \emph{Bockstein map}.

For example, continuing with earlier notation,  let $Y$ be a connected topological space with base point $y_0$ and $\pi=\pi_1(Y,y_0)$. Then for the $(R,\Z[\pi])$-module $R^d_\alpha $, and with $C_* = C_*(\wt{Y})^t$, we have $\op{BS} \colon H^1(Y; Q/R \otimes_{R} R^d_\alpha) \to H^2(Y;R^d_\alpha)$.

\subsection{Definition of the twisted Blanchfield pairing of a 3-manifold}\label{section:blanchfield}
In the following let $M$ be a closed 3-manifold and write $\pi=\pi_1(M)$. Let $\a\colon \pi\to \op{GL}(R,d)$ be a representation over a commutative domain $R$. As above denote the quotient field of $R$ by $Q$.

\emph{Assume that $H_1(M;R^d_\alpha)$ is $ R$-torsion}. By Poincar{\'e} duality and universal coefficients, this is equivalent to saying that  $H^*(M;Q \otimes_{R} R^d_\alpha)=0$, which in turn implies that the  Bockstein map $\op{BS} \colon H^1(M; Q/R \otimes_{R} R^d_\alpha)\to H^2(M;R^d_\alpha)$ is an isomorphism.
Let $\Psi$ be the composition of the following maps
\[ \xymatrix@C1.2cm@R0.5cm{H_1(M;R^d_\alpha)\ar@/_2pc/[rddr]_\Psi\ar[r]\ar[r]^-{\PD^{-1}}_-{\cong} & H^2(M;R^d_\alpha)\ar[r]^-{\BS^{-1}}_-\cong& H^1(M; Q/R \otimes_R R^d_\alpha) \ar[d]^-(0.45){\kappa}\\
&& H_1(\hom_{R}(C_*(M;R_{\alphas}^d),Q/ R)^t)\ar[d]^(0.45){\ev}\\
&& \hom_{R}(H_1(M;R^d_{\alphas}),Q/ R)^t.}\]
From the adjoint of $\Psi$ we obtain a form
\[ \ba{rcl} \Bl \colon H_1(M;R^d_{\alphas})\times H_1(M;R^d_{\alpha})&\to& Q/ R\\
(a,b)&\mapsto &\Psi(b)(a)\ea\]
which is referred to as the \emph{twisted Blanchfield pairing} of $(M,\alpha)$.
This pairing is sesquilinear over $ R$, in the sense that $\Bl(pa,qb)=p\Bl(a,b)\ol{q}$ for any $a,b\in H_1(M;R^d_\alpha)$ and $p,q\in  R$.

If $\a$ is unitary, that is if $\a=\alphas$, then we obtain a pairing
\[ \ba{rcl} \Bl^{\alpha}\colon H_1(M;R^d_\alpha)\times H_1(M;R^d_{\alpha})&\to& Q/ R\\
(a,b)&\mapsto &\Psi(b)(a).\ea\]

\begin{prop}
Suppose that $M$ is a closed 3-manifold, $\alpha$ is unitary, and that $R$ is a PID.
Then $\Bl^{\alpha}$ is nonsingular and hermitian.
\end{prop}

\begin{proof}
  The proposition was proved in \cite{Powell:2016-1}.  In particular we refer to this paper for the fact that $\op{Bl}^{\alpha}$ is hermitian.  Nonsingularity follows from the fact that each of the maps in the above diagram are isomorphisms. In particular, the composition $\ev \circ \kappa$ is an isomorphism by the universal coefficient theorem: since $R$ is a PID, $Q/R$ is an injective $R$-module and so higher $\Ext$ terms such as $\Ext^1_R(H_1(M;R^d_{\alpha}),Q/R)$ terms vanish.
\end{proof}

\section{The twisted Blanchfield pairing via the symmetric chain complex}\label{section:algebraic-defn-TBF}

In order to compute the twisted Blanchfield pairing explicitly, we will use a formula describing the pairing of two homology classes in terms of chain level representatives.

Let $\pi$ be a finitely presented group, let $R$ be a commutative domain with involution, let $Q$ be its field of fractions, and let $V$ be an $(R,\Z[\pi])$-bimodule.
Let $(C_*,\Phi) $ be a 3-dimensional symmetric Poincar\'{e} chain complex.

Define $V^* := \Hom_R(V,R)^t$, the $R$-dual, converted into a left $R$-module using the involution.
The right $\Z[\pi]$-module structure of $V^*$ is defined via $(f\cdot g)(v) = f(v\ol{g})$, where $f \in V^*$, $g \in \Z[\pi]$ and $v \in V$.
After tensoring a chain complex and its dual with $V$, the boundary, coboundary and symmetric structure maps $f=\partial, \partial^*$ or $\Phi_s$ become $\Id \otimes f$, however we usually omit $\Id \otimes$ from the notation.

For future reference we record the following elementary lemma.

\begin{lemma}\label{lem:isoduals}
Let $R$ and $A$ be rings with involution.
Let $V$ be an $(R,A)$-bimodule and let $W$ be a free finitely generated $A$-module.  Define $V^* := \Hom_R(V,R)^t$ and $W^* := \Hom_{A}(W,A)^t$.  Define the right $A$-module structure on $V^*$ by $(f\cdot s)(v) = f(v\ol{s})$, where $f \in V^*$, $v \in V$ and $s \in A$.  Then the map
\[ \ba{rcl} (V^*\otimes_{A} W^*)&\to & (V\otimes_A W)^* \\
(\phi\otimes f)&\mapsto & (v\otimes w\mapsto \phi(v\cdot f(w))\ea \]
is an isomorphism of left $R$-modules.
\end{lemma}

Suppose that the right $\zpx$-module structure  on $V$ is given by a representation $\a \colon \pi \to \Aut(V)$. Suppose also that $V$ has an $R$-sesquilinear, nonsingular inner product $\ll \cdot, \cdot \rr$.
Then $\a$ is called \emph{unitary} (with respect to $\a$) if $\ll v\a(g),w\a(g)\rr = \ll v,w \rr$ for all $v,w \in V$ and $g \in \pi$. (Equivalently, $\a$ is unitary if the inner product satisfies the conditions of Section~\ref{section:kronecker}.)

The data of an inner product with respect to which $\a$ is unitary is equivalent to an isomorphism of $(R,\zpx)$-bimodules $\Theta \colon V \toiso V^*$.  From now on we require that $V$ is always equipped with such an isomorphism $\Theta$, or equivalently we require that the representation $\a$ be unitary. For example, when $V=R^n$ with the standard hermitian inner product $\ll v,w \rr= v \cdot \ol{w}$ the requirement that $\a \colon \pi \to \GL(V)$ is unitary coincides with the usual notion of a unitary representation.
In this case the inner product $\ll \cdot, \cdot \rr$ is equivalent to the pairing in equation (\ref{equation:pairing-for-evaluation-map}), which has the same notation.

For an $R$-module $N$, let $TN$ denote the maximal torsion submodule $\{n \in N \,|\, rn=0 \text{ for some } r \in R\sm\{0\}\}$.
The next definition is based on~\cite[Page~185]{Ranicki:1981-1}.  See also \cite[Proposition~3.4.1]{Ranicki:1981-1} for the precise relationship between symmetric complexes and linking pairings.

\begin{definition}[Chain level twisted Blanchfield pairing]\label{defn:chain-level-blanchfield}
The twisted Blanchfield pairing of a 3-dimensional symmetric chain complex $(C_*, \Phi)$,
\[\wt{\Bl} \colon TH^2(V \otimes_{\Z[\pi]} C) \times TH^2(V \otimes_{\Z[\pi]} C) \to Q/R,\]
is defined as follows.  For $[x], [y] \in TH^2(V \otimes_{\Z[\pi]} C)$, let
\[\wt{\Bl}([y],[x]) = \smfrac{1}{s} \ol{z(\Phi_0(x))}\]
where $x, y \in V \otimes_{\Z[\pi]} C^2$, $z \in V \otimes_{\Z[\pi]} C^1$ and $\partial^*(z) = sy$ for some  $s \in R\sm\{0\}$.
To evaluate $z$ on $\Phi_0(x)$ use the image of $z$ under the isomorphisms
\[V \otimes_{\Z[\pi]} C^1 \xrightarrow{\Theta \otimes \Id} V^* \otimes_{\zpx} C_1^* \xrightarrow{\text{Lemma }\ref{lem:isoduals}} (V \otimes C_1)^*\]

For a symmetric \emph{Poincar\'{e}} complex, we can also define the Blanchfield pairing on homology:
\[\Bl \colon TH_1(V \otimes_{\Z[\pi]} C) \times TH_1(V \otimes_{\Z[\pi]} C) \to Q/R\]
via $\Bl([u],[v]) := \wt{\Bl}([\Phi_0]^{-1}([u]),[\Phi_0]^{-1}([v]))$.
\end{definition}

\begin{proposition}\label{Prop:chainlevelBlanchfield}\cite{Powell:2016-1}~
\begin{enumerate}[font=\normalfont]
\item[(i)] The twisted Blanchfield pairing of Definition~\ref{defn:chain-level-blanchfield} is well-defined and sesquilinear in $R$.
\item[(ii)] If the $\Z[\pi]$ structure on $V$ is unitary, then the twisted Blanchfield pairing of a symmetric complex is hermitian, i.e.\ $\wt{\Bl}([y],[x])= \ol{\wt{\Bl}([x],[y])}$.
\item[(iii)] Whenever $C_*=C_*(M, R_{\a}^d)$ is the symmetric Poincar{\'e} chain complex of a closed 3-manifold, the above definition of a twisted Blanchfield pairing coincides with that of Definition~\ref{section:blanchfield}.  Hence when $R$ is a PID, the pairing is nonsingular.
\end{enumerate}
\end{proposition}

Note that, in the conventions of Definition~\ref{defn:dual-complex}, we have $\partial^* = \delta \colon C^1 \to C^2$, in particular there is no minus sign here.  Thus the two definitions of the twisted Blanchfield pairing agree, and not just up to sign.  On the other hand, the sign depends on the choice of fundamental class of $M$, and in applications we will make this choice somewhat arbitrarily. Indeed, in Remark~\ref{Rmk:relativefundclass}, we already did so.  In addition, in our application to knot concordance, we only ever need to show that a Blanchfield pairing is nonvanishing, so we never need to determine its sign.

\begin{remark}\label{remark:phi0dual}
The proof of the above proposition in \cite{Powell:2016-1} also shows that one may use $\Phi_0^*$ instead of $\Phi_0$ to compute the twisted Blanchfield pairing. This may be exploited in computations, when the map $\Phi_0^*\colon C^2 \to C_1$ may be simpler than $\Phi_0 \colon C^2 \to C_1$.  For example this is the case for the formulae in Theorem~\ref{trotterdelta0}.
\end{remark}

\subsection{Computing the twisted Blanchfield pairing}\label{section:computing-TBP}

Now we state our algorithm for computing twisted Blanchfield pairings of the knot 0-surgery $M_K$, starting from a reduced diagram with $c$ crossings for $K$ and a unitary map $\a: \pi \to \Aut_R(V)$ for some commutative domain $R$ and $V=R^k_{\a}$.

From Theorem~\ref{Thm:mainchaincomplex}, we have  explicit formulae for $C_*(M_K, \Z[\pi_1(M_K)])$:
\[\underset{C_3 \cong \langle h^3_1,h^3_{s} \rangle}{\underbrace{ \bigoplus_{2}{}\,\Z[\pi] }}\,\,
\xrightarrow{\partial_3}\,\,
\underset{C_2 \cong \langle h^2_1,\dots,h^2_c,h^2_s \rangle }{\underbrace{\bigoplus_{c+1}{}\,\Z[\pi] }}\,\, \xrightarrow{\partial_2}\,\,
\underset{C_1 \cong \langle h^1_1,\dots,h^1_c \rangle }{\underbrace{\bigoplus_{c}{}\,\Z[\pi]  }}\,\, \xrightarrow{\partial_1}\,\,
\underset{C_0 \cong \langle h^0 \rangle}{\underbrace{ \Z[\pi] }},
\]
where with respect to the above bases we have
\[\ba{rcl}
           \partial_3=  \left(\begin{array}{cc}
            w_1 & -u_1 \\
            \vdots & \vdots \\
            w_c & -u_c \\
            0 & g_1-1
            \end{array}\right)^T,

\partial_2 = \left(
                 \begin{array}{ccccc}
                   \left(\partial r_1/\partial g_1\right) &  & \hdots &  & \left(\partial r_1/\partial g_c\right)\\
                    &  &  &  &  \\
                   \vdots &  & \ddots &  & \vdots \\
                    &  &  &  &    \\
                   \left(\partial r_c/\partial g_1\right) &  & \hdots &  & \left(\partial r_c/\partial g_c\right) \\
                   \left(\partial \ell/\partial g_1\right) &  & \hdots &  & \left(\partial \ell/\partial g_c\right)\\
                 \end{array}
               \right),
\partial_1 = \left(
               \begin{array}{c}
                 g_1 - 1 \\
                  \\
                 \vdots \\
                 \\
                  g_c - 1\\
               \end{array}
             \right). \ea\]
  The elements $\{r_i\}$ were defined in Proposition~\ref{prop:wirtingerpresentation}, the $\{u_i\}$ and the $\{w_i\}$ were defined in Theorem~\ref{Thm:mainchaincomplex}, and $\ell$ was given in Equation~\ref{eqn:longitudeword}.

 We also have the relevant piece of the symmetric structure map, given by
 \[\Phi_0^*= \ol{\Phi_0}^T= \left(\begin{array}{ccccc}
 -w_1^{-1}\frac{\partial w_1}{\partial g_1} + u_1^{-1}\frac{\partial u_1}{\partial g_1} & & \hdots & & -w_1^{-1}\frac{\partial w_1}{\partial g_c} + u_1^{-1}\frac{\partial u_1}{\partial g_c} \\
 & & & &  \\
 \vdots & & \ddots & & \vdots  \\
 & & & &  \\
 -w_c^{-1}\frac{\partial w_c}{\partial g_1} + u_c^{-1}\frac{\partial u_c}{\partial g_1} & & \hdots & & -w_c^{-1}\frac{\partial w_c}{\partial g_c} + u_c^{-1}\frac{\partial u_c}{\partial g_c} \\
 g_1^{-1} & & \hdots & & 0
 \end{array} \right): C^2 \to C_1.\]

Form the chain complexes $Y_*:= V \otimes_{\zpx} C_*(M_K, \Z[\pi])$, so $H_*(Y)=H_*(M_K,V).$
It is natural to take our original bases for $C_*(M_K, \Z[\pi])$, choose an $R$-basis for $V$, and work with the bases thereby obtained for $Y_*$ and, dually, $Y^*$. We then represent each boundary, coboundary or symmetric structure map as a matrix with respect to these bases and by a mild abuse of notation continue to refer to them by the same names.
 Note that for example the matrix representative for  $\partial_2: Y_2 \to Y_1$ is a $\dim(V)\cdot (c+1) \times \dim(V) \cdot c$ matrix with entries in $R$, by our convention that matrices act on row vectors from the right.

So we have explicit formulae for all maps in the following diagram.
\[\xymatrix @C+1cm{
Y^0 \ar[r]^{\partial^*_1} & Y^1 \ar[r]^{\partial^*_2} & Y^2 \ar[d]^{\Phi_0^*} \ar[r]^{\partial^*_3} & Y^3 \\
Y_3 \ar[r]_{\partial_3} & Y_2 \ar[r]_{\partial_2} & Y_1 \ar[r]_{\partial_1} & Y_0}\]
%
%
%
Now suppose that we would like to compute the twisted Blanchfield pairing on elements $[x],[y] \in TH_1(Y)$. Recall that we must first find chain level representatives of $(\Phi_0^*)^{-1}([x])$ and $(\Phi_0^*)^{-1}([y])$ in $H^2(Y)$. That is, recalling once more that all maps act on the right (Conventions~\ref{conventions2}), we must find $A,B \in Y^2$ and $a,b \in Y_2$ such that
\begin{equation}\label{equation:liftingchains}
A \cdot \Phi_0^*= x+ a \cdot \partial_2 \quad \text{ and } \quad B \cdot \Phi_0^*= y+ b \cdot \partial_2.
\end{equation}
We then need $Z \in Y^1$ and $s \in R$ such that
$Z\cdot \partial_2^*= s A.$
Having these, we can compute
\begin{equation}\label{equation:computing-form}
\Bl([x], [y])\,\,=\,\, \wt{\Bl}([A], [B])\,\,=\,\, \smfrac{1}{s}\ol{ Z(B \cdot \Phi_0^*)} \in Q/R.
\end{equation}

\subsection{Computational considerations}\label{section:matrix-moves-computing-TBF}

Directly finding the various $A,B, a, b, Z$, and $s$ as outlined above is generally quite difficult, and in practice one should instead change basis via the following algorithm, assuming now that $R$ is a Euclidean domain. Note that in fact these computations are made reasonable by use of a computer algebra system such as Maple, or any program capable of straightforward column and row operations on matrices. We provide one such program on our personal research websites (see the end of the paper for the url addresses at the time of writing), The program is also available from the authors on request.  We also used the computer to evaluate Fox derivatives and calculate the image of our matrices under the representations that we use.

First we need to simplify the matrices representing the coboundary maps so that we can compute the cohomology $H^2(Y)$.
We may perform row and column operations to obtain a matrix in reduced echelon form, but whose pivots need only be nonzero, and need not be the entry~$1$.  In the case of a nonsingular square matrix we obtain a diagonal matrix.  The Gaussian algorithm cannot be used in the standard way since division is not permitted.  Nevertheless we can perform operations until the $(1,1)$ entry contains the greatest common divisor of the $1$st row and the $1$st column.  This can then be used to clear the other entries of that row and column using further row and column operations.  Proceed inductively to arrange that the $(i,i)$ entry contains the greatest common divisor of the $i$th row and the $i$th column, and then clear using that entry.

Let $\delta_1$, $\delta_2$, $\delta_3$ and $F_0^*$ be matrix representatives for $\partial_1^*$, $\partial_2^*$, $\partial_3^*$ and $\Phi_0^*$ respectively, with respect to our original basis.  Explicit matrices were given in the previous section.    Perform the simplification process above on $\delta_3$ first.  Then perform the process on $\delta_2$.  One can then read off a simple presentation for the cohomology $H^2(Y) = H^2(M_K;V)$.    However during the process of row and column operations we must perform simultaneous operations on the four matrices above, since change of basis on $C^1$ changes the domain of $\delta_2$ and the codomains of $\delta_1$ and $F_0^*$, while change of basis on $C^2$ changes the domains of $\delta_3$ and $F_0^*$, and the codomain of $\delta_2$. For example, a change of basis of $C^2$ resulting in the operation $\delta_3 \mapsto L\delta_3$, where $L$ is an elementary matrix in $[\GL(R),GL(R)]$, should be accompanied by the changes $\delta_2 \mapsto \delta_2L^{-1}$ and $F_0^* \mapsto LF_0^*$.  Here recall again that matrices act on the right, as described in Conventions~\ref{conventions2}.

At the end of this process, the matrix $\delta_2$ has become a diagonal matrix, perhaps with zero entries, plus a column of zeroes.  Note that this also gives such a nearly-diagonal form for the matrix representing $\partial_2\colon Y_2 \to Y_1$. This substantially eases the computations required by Equation~\ref{equation:liftingchains},  while making those of Equation~\ref{equation:computing-form} immediate.

\section{Twisted Blanchfield pairings and knot concordance}\label{section:TBFs-and-knot-concordance}

In this section, we obtain conditions on twisted Blanchfield forms of slice knots, namely that they must be metabolic for certain representations.  Moreover, the metabolisers correspond to potential discs.

We pause to establish notational conventions for the rest of the paper.
Let $K$ be an oriented knot in $S^3$. There are associated 3-manifolds $X_K= S^3 \smallsetminus \nu(K)$, the knot exterior, and $M_K= S^3_0(K)$, the 3-manifold obtained by zero framed surgery along $K$.

For $k \in \N$, we let $X_k= X_k(K)$ and $M_k=M_k(K)$ denote the canonical cyclic $k$-fold covers of $X_K$ and $M_K$, respectively. We also let $\Sigma_k= \Sigma_k(K)$ denote the $k$-fold cyclic branched cover of $S^3$ along $K$.  Observe also that if $W$ is any $4$-manifold with $\partial W= M_K$ and $H_1(M_K;\Z) \to H_1(W;\Z)$ an isomorphism, then $W$ has a canonical $k$-fold cyclic cover $W_k$ for $k \in \N$.

The results of this section are particularly indebted to discussions with Stefan Friedl.

\subsection{The solvable filtration}\label{section:filtration}

For the convenience of the reader, we briefly recall the definition of the solvable filtration of knot concordance.  Our slice obstructions also obstruct knots from lying in certain stages of the filtration, so we will state our theorems in terms of the filtration, in order to give the strongest possible statements.

\begin{defn}[Derived series]~
Let $G$ be a group.
\begin{enumerate}[font=\normalfont]
\item[(a)] The \emph{$m^{th}$ derived subgroup} of $G$ is defined recursively via $G^{(0)}:=G$ and $G^{(m)}:= [G^{(m-1)}, G^{(m-1)}]$ for $m \geq 1$.
\item[(b)] Let $\mathcal{S}= (S_i)_{i \in \N}$, where each $S_i$ is either $\Q$ or $\Z_n$ for some $n \in \N$. Define $G_{\mathcal{S}}^{(0)}= G$ and define the \emph{$(m+1)^{th}$ $\mathcal{S}$-local derived subgroup} of $G$ iteratively by
\[G_{\mathcal{S}}^{(m+1)}= \ker\left\{G_\mathcal{S}^{(m)} \to G_\mathcal{S}^{(m)}/ [G_\mathcal{S}^{(m)},G_\mathcal{S}^{(m)}] \to \left(G_\mathcal{S}^{(m)}/ [G_\mathcal{S}^{(m)},G_\mathcal{S}^{(m)}]\right) \otimes S_{m+1} \right\}.
\]
 \end{enumerate}
\end{defn}

 We have the following generalisation of slice disc exteriors.

\begin{defn}\label{Defn:filtration}\cite{Cochran-Orr-Teichner:1999-1}
A knot $K$  is called \emph{$m$-solvable} for $m \in \N \cup \{0\}$ if there exists a compact, oriented, spin $4$-manifold $W$ with $\partial W = M_K$  satisfying the following conditions.
\begin{enumerate}
\item The inclusion induced map $H_1(M_K;\Z) \to H_1(W;\Z)$ is an isomorphism.
\item There is a collection of embedded surfaces $L_1, \dots, L_g$ and $D_1, \dots, D_g$ with trivial normal bundles forming a basis for $H_2(W;\Z)$ that are pairwise disjoint, except that $L_i$ and $D_i$ intersect transversally in a single point, for each $i-1,\cdots,g$.
\item $\im\left(\pi_1(L_i) \to \pi_1(W)\right)$ and $\im\left(\pi_1(D_i) \to \pi_1(W)\right)$ are both contained in $\pi_1(W)^{(m)}$ for $i=1, \dots, g$.
\end{enumerate}
We call $W$ an \emph{$m$-solution} for $K$.
We say that $K$ is \emph{$m.5$-solvable} if $M_K$ bounds some $m$-solution $W$ with the additional property that $\im\left(\pi_1(L_i) \to \pi_1(W)\right)$ is contained in $\pi_1(W)^{(m+1)}$.  In this case we call $W$ an $m.5$-solution.
\end{defn}
\noindent We remark that a slice disc exterior is an $m$- and $m.5$-solution for all $m\geq 0$.

\subsection{Branched covers and linking forms}\label{section:branched}

Note that $H_1(\Sigma_k;\Z)$ has a natural $\Z_k$ action, and we can therefore view $H_1(\Sigma_k;\Z)$ as a $\Z[\Z_k]$-module.
If $H_1(\Sigma_k;\Z)$ is finite, then there exists  a nonsingular linking form
\[\lambda_k\colon H_1(\Sigma_k;\Z)\times H_1(\Sigma_k;\Z)\to \Q/\Z\]
with respect to which $\Z_k$ acts via isometries.

We say that $P\subset H_1(\Sigma_k;\Z)$ is a \emph{metaboliser} of the linking form if $P$ is a $\Z[\Z_k]$-submodule of $H_1(\Sigma_k;\Z)$ such that $\lambda_k(P,P)=0$ and such that $|P|^2=|H_1(\Sigma_k;\Z)|$. We emphasise that in our definition of a metaboliser we require that $P$ is not just a subgroup, but a  $\Z[\Z_k]$-submodule of $H_1(\Sigma_k;\Z)$ --- that is, $P$ is preserved by the map on homology induced by the covering transformation.

We recall the following well-known lemma (see for example \cite{Casson-Gordon:1986-1}, \cite{Go78} or \cite[Section~2.5]{Friedl:2003-4} in the case when $K$ is slice).

\begin{lemma}\label{lem:linkmetabolic}\cite[Proposition~9.7]{Cochran-Orr-Teichner:1999-1}
Let $K$ be an oriented knot with 1-solution $W$ and $k$ be a prime power.
Let $P':= \ker\{H_1(X_k; \Z) \to H_1(M_k;\Z)\to H_1(W_k;\Z)\}$ and let $P$ be the image of $P'$ under the inclusion induced map $H_1(X_k; \Z) \to H_1(\Sigma_k; \Z)$. Then $P$ is a metaboliser of $\lambda_k$.
\end{lemma}

\subsection{Metabelian representations}\label{section:reps}

Let $H$ be a $\zt$-module. Let $n\in \N$ and denote by $\zeta_n$ a primitive $n$th root of unity.
Given a  character $\chi\colon H\to H/(t^k-1)\to \Z_n$, define $\alpha(k,\chi)$ to be the representation given by
 $$ \begin{array}{rcl}  \Z\ltimes H &\to& \Z\ltimes H/(t^{k}-1) \to \GL(k,\Z[\zeta_n]\tpm) \\
  & & \\
    (j,h) &\mapsto &
\begin{pmatrix}
 0& \dots &0&t \\
 1&0& \dots &0 \\
\vdots &\ddots &  & \vdots\\
    0&\dots &1&0 \end{pmatrix}^j
\begin{pmatrix}
\zeta_q^{\chi(h)} &0& \dots&0 \\
 0&\zeta_q^{\chi(th)} & \dots& 0\\
\vdots&&\ddots & \vdots\\ 0&0&\dots &\zeta_1^{\chi(t^{k-1}h)} \end{pmatrix}. \end{array} $$
It is straightforward to verify that this defines a unitary representation of $\Z\ltimes H$.

Now we return to the study of knots.  Let $K$ be an oriented knot in $S^3$.
Note that $H^1(M_K;\Z)=\Z$, so $H^1(M_K;\Z)$ has a unique generator (a priori up to sign, but the sign is determined by the orientation of $K$) that we denote by $\phi$.  We also identify $H^1(M_K;\Z) \cong \Hom(\pi_1(M_K),\Z)$, and denote the image of $\phi$ under this identification also by $\phi$.

Consider the Alexander module  $H_1(M_K;\zt)$, which we shall denote by $\mathcal{A}$.
Note that~$\mathcal{A}$ is isomorphic to the usual Alexander module of~$K$, with $M_K$ replaced by the exterior $X_K$.
There exists a canonical isomorphism
$\mathcal{A}/(t^k-1)\to H_1(\Sigma_k;\Z)$ (see for example \cite[Corollary~2.4]{Friedl:2003-4} for details).
Now let $\mu \in \pi_1(M_K)$ be an element  with $\phi(\mu)=1$.
Note that for any $g\in \pi_1(M_K)$ we have $\phi(\mu^{-\phi(g)}g)=0$, in particular $\mu^{-\phi(g)}g$
represents an element in the abelianisation of $\ker(\phi)$. We identify $\ker(\phi)_{ab}$ with $\mathcal{A}$.
Then we have a well-defined map
\[ \pi_1(M_K) \to \Z \ltimes \mathcal{A} \to \Z \ltimes \mathcal{A}/(t^k-1), \]
where the first map is given by sending $g\in \pi_1(M_K)$ to $(\phi(g),[\mu^{-\phi(g)}g])$.
 Here $n\in \Z$ acts on $\mathcal{A}$ and on $\mathcal{A}/(t^k-1)$ via multiplication by~$t^n$.
We refer to \cite{Friedl:2003-4} and \cite{BF08} for details.

Now let $k \in \N$ and let $\chi \colon H_1(\Sigma_k;\Z)\to \Z_n$ be a character.
Denote the induced character $\mathcal{A} \to \mathcal{A}/(t^k-1)=H_1(\Sigma_k;\Z) \to \Z_n$ also by $\chi$.
Then we have a unitary representation
\[ \pi_1(M_K) \to  \Z\ltimes \mathcal{A} \xrightarrow{\alpha(k,\chi)} \GL(k,\Z[\zeta_n]\tpm). \]
Finally, denote the quotient field of $\Z[\zeta_n]$ by $\mathbb{F} := \Q(\zeta_n)$ and abuse notation to also denote the  unitary representation
\[\pi_1(M_K) \to  \Z\ltimes \mathcal{A} \to \GL(k,\Z[\zeta_n]\tpm) \to \GL(k,\ft)\]
by $\alpha(k,\chi)$.

\subsection{Highly solvable knots have metabolic twisted Blanchfield pairings}\label{section:TBF-metabolic-2.5-solvable}

As before, $K$ is a knot with 0-surgery $M_K$.
Let $\mathbb{F}$ be a field with (a perhaps trivial) involution.

\begin{definition}
A representation $\a \colon \pi_1(M_K) \to \GL(k,\ft)$ is called \emph{acyclic} if $\mathbb{F}(t) \otimes_{\ft} H_i(M_K;\ft^k_{\a}) = H_i(M_K;\mathbb{F}(t)^k_{\a}) = 0$ for all $i$.
\end{definition}

For an acyclic representation we have that $H_i(M_K;\ft^k)$ is $\ft$-torsion.
Now let $\a\colon \pi_1(M_K)\to \GL(k,\ft)$ be a unitary acyclic representation, where we set $R=\ft$ and $V=R^k_{\a}$.
We can then consider the twisted Blanchfield pairing
\[ \bla\colon H_1(M_K;\ft^k_{\a})\times H_1(M_K;\ft^k_{\a})\to \F(t)/\ft\]
and recall that we say that $\bla$
is \emph{metabolic} if there exists a submodule $P\subset H_1(M_K;\ft^k_{\a})$ with
\[ P=P^{\perp}:=\{ x\in H_1(M_K;\ft^k_{\a}) \, |\, \bla(x,y)=0 \mbox{ for all }y\in P\}.\]
The \emph{twisted Alexander polynomial}  of $(K,\a)$ is
\[\Delta_{K}^\a := \ord(H_1(M_K;\ft^k_{\a})).\]
Metabolic twisted Blanchfield pairings correspond to twisted Alexander polynomials that are norms, as explained in the next lemma.
Once we have proven that slice knots have metabolic twisted Blanchfield pairings, we will recover the result from \cite{Kirk-Livingston:1999-2}, that the twisted Alexander polynomials of slice knots are norms.

\begin{lemma}\label{lem:alexnorm}
If $\bla$ is metabolic, then $\Delta_{K}^\a=f(t)\ol{f(t)}$ for some $f(t)\in \ft$.
\end{lemma}

\begin{proof}
Let $P\subset H_1(M_K;\ft^k_{\a})$ be a metaboliser. It is straightforward to verify
that
\begin{align*}
0 \to P \to  H_1(M_K;\ft^k_{\a}) &\to \Hom(P,\ft)^t  \to  0 \\
  x &\mapsto (y\mapsto \bla(x,y))
 \end{align*}
is a short exact sequence. It follows immediately
that
\[\Delta_{K}^\a=\ord(H_1(M_K;\ft^k_{\a}))=\ord(P)\cdot \ord(\ol{\Hom(P,\ft)})=\ord(P)\cdot \ol{\ord(P)}.\]
\end{proof}

The following result of Casson and Gordon~\cite[Corollary to Lemma~4]{Casson-Gordon:1986-1} guarantees that certain metabelian representations are acyclic.  An alternative proof was given later in~\cite{Friedl-Powell:2010-1}.

\begin{lemma}\label{lem:CGacyclic}
Suppose that $K$ is an oriented knot, $k \in \N$, and $\chi\colon H_1(\Sigma_k;\Z)\to \Z_n$ is a nontrivial  character of prime power order $n$. Let $\F= \Q(\zeta_n)$ as above and let $\alpha(k,\chi)\colon \pi_1(M_K) \to \GL(k,\ft)$ be defined as above. Then $\alpha(k,\chi)$ is acyclic.
\end{lemma}

The next proposition generalises the fact that the ordinary Blanchfield pairing of a slice knot is metabolic
c.f.\ \cite{Ke75}, \cite[Proposition~2.8]{Let00}.

\begin{proposition} \label{prop:metabolic}
Let $K$ be an oriented knot, let $W$ be a compact oriented $4$-manifold with boundary $M_K=\partial W$, and let $n$ be a prime power.
Consider the representation $\a=\a(k,\chi)\colon \pi_1(M_K)\to \GL(\ft,k)$ coming from $\chi\colon H_1(\Sigma_k;\Z)\to \Z_n$ as detailed above, with $\mathbb{F}:= \Q(\zeta_n)$ for some prime power~$n$.
Assume that $\a$ extends over $W$ to a representation $\pi_1(W)\to \GL(\ft,k)$, and that the following sequence is exact:
\[TH_2(W, M_K;\ft^k_{\a}) \to TH_1(M_K;\ft^k_{\a}) \xrightarrow{i} TH_1(W;\ft^k_{\a}).\]
Then the Blanchfield pairing $$\bla\colon H_1(M_K;\ft^k_{\a})\times H_1(M_K;\ft^k_{\a})\to \F(t)/\ft$$ is metabolic with metaboliser $P:=\ker i$.
\end{proposition}

\begin{proof}
We first claim that $P$ is self-annihilating with respect to $\bla$.
This is a standard argument~\cite[Proposition~2.8]{Let00}, \cite[Theorem~4.4]{Cochran-Orr-Teichner:1999-1}, since
\[TH_2(W, M_K;\ft^k_{\a}) \to TH_1(M_K;\ft^k_{\a}) \xrightarrow{i} TH_1(W;\ft^k_{\a})\]
is exact.
However,  since $\ft$ is a PID we also have that $P^{\perp} \subseteq P$ and so $P$ is a metaboliser, as in the proof of \cite[Theorem~4.4]{Cochran-Orr-Teichner:1999-1}. See also \cite[Theorem~2.4]{Hillman:2012-1-second-ed}.
\end{proof}

In the next section we will show that the hypotheses of Proposition~\ref{prop:metabolic} hold when~$W$ is a $2$-solution for the knot~$K$.

\subsection{Extending representations over slice disc exteriors}

Most of the ideas and techniques of finding  representations satisfying the conditions set out in Proposition \ref{prop:metabolic}
 go back to the seminal work of Casson and Gordon \cite{Casson-Gordon:1986-1}.
We follow the approach taken in \cite{Friedl:2003-4}, which was inspired by the work of Letsche \cite{Let00} and Kirk-Livingston \cite{Kirk-Livingston:1999-2}.

Recall that given a character $\chi\colon H_1(\Sigma_k;\Z)\to \Z_n$ of prime power order and a choice of $\mu \in \pi_1(M_K)$ such that $\phi(\mu) =1$, we defined, in Section~\ref{section:reps}, a metabelian unitary representation  $\a{(k,\chi)} \colon \pi_1(M_K)  \to  \GL(k,\Q(\zeta_n)\tpm)$. We henceforth suppress our choice of $\mu$ with $\phi(\mu)=1$, since it does not affect the argument.

\begin{proposition} \label{prop:extend}
Let $K$ be an oriented knot with a $2$-solution $W$.
Let $k$ be a prime power.
 Let $P$ be the image of $\ker\left(H_1(X_k; \Z)\to H_1(M_k;\Z) \to H_1(W_k;\Z)\right)$ in $H_1(\Sigma_k)$, as in Lemma~\ref{lem:linkmetabolic}.
Then $P$ is a metaboliser  of $\lambda_k$. Now, let $\chi$ be a nontrivial character $\chi\colon H_1(\Sigma_k;\Z)\to \Z_{q^a}$ of prime power order $q^a$ that vanishes on~$P$. Then for some $b \geq a$,
 letting $\chi_b \colon H_1(\Sigma_k) \to \Z_{q^a} \hookrightarrow \Z_{q^b}$ be the composition of $\chi$ with the natural inclusion $ \Z_{q^a} \hookrightarrow \Z_{q^b}$, we have that
\begin{enumerate}[font=\normalfont]
\item[(i)] $\a(k,\chi_b)$ extends to a representation $\pi_1(W)\to \GL(\Q(\zeta_{q^b})[t^{\pm1}],k)$ that we call $\a$, and
\item[(ii)] the following sequence is exact, where $\F= \Q(\zeta_{q^b})$:
\begin{align}\label{torsionexact}
TH_2(W, M_K;\ft^k_{\a}) \to TH_1(M_K;\ft^k_{\a}) \to TH_1(W;\ft^k_{\a}).
\end{align}
\end{enumerate}
\end{proposition}

The proof of the  proposition builds on the ideas of Casson-Gordon \cite{Casson-Gordon:1986-1} and Cochran-Orr-Teichner~\cite{Cochran-Orr-Teichner:1999-1}. We also refer to Letsche~\cite{Let00} and Friedl~\cite{Friedl:2003-4} for more information.

\begin{proof}
By Lemma \ref{lem:linkmetabolic}, the submodule  $P$ is a metaboliser for the linking form.
The character $\chi$ extends to a character $\chi_b \colon H_1(W_k;\Z)\to \Z_{q^b}$ for some $b\geq a$.
Now write $H:= H_1(W;\zt)$.
Similarly to the knot exterior case we then have a canonical map $\pi_1(W)\to \Z\ltimes H$ and an isomorphism $H/(t^k-1)\cong H_1(W_k;\Z)$.  Here we use the image of the element $\mu$ in $\pi_1(W)$ to determine the splitting into the semi-direct product, and note that the splittings of $\pi_1(M_K)/\pi_1(M_K)^{(2)}$ and $\pi_1(W)/\pi_1(W)^{(2)}$ are compatible.
We can therefore consider the representation
\[ \pi_1(W)\to \Z\ltimes H \xrightarrow{\a(k,\chi_b)} \GL(k,\Q(\zeta_{q^b})\tpm)\]
which extends the representation \[\a(k,\chi)\colon \pi_1(M_K)\to \GL(k,\Q(\zeta_{q^a})\tpm)\to \GL(k,\Q(\zeta_{q^b})\tpm)\]
cf.\ \cite[Lemma~4.3]{Friedl:2003-4}.

For the rest of the proof let $\mathbb{F} = \Q(\zeta_{q^b})$.
It remains to show that the sequence \[TH_2(W, M_K;\ft^k_{\a}) \xrightarrow{\delta|_T} TH_1(M_K;\ft^k_{\a}) \xrightarrow{i|_T} TH_1(W;\ft^k_{\a})\]
is exact.
First, the fact that $\im(\delta|_T)\subseteq \ker(i|_T)$ follows immediately from the exactness of $H_2(W, M_K;\ft^k_{\a}) \xrightarrow{\delta} H_1(M_K;\ft^k_{\a}) \xrightarrow{i} TH_1(W;\ft^k_{\a})$.
On the other hand, let $x \in \ker(i|_{T}) \subset \ker(i)= \im(\delta)$, to show that $x \in \im(\delta|_T)$. It is enough to show that
$$H_2(W, M_K;\ft^k_{\a})\cong TH_2(W, M_K;\ft^k_{\a}) \oplus F,$$ where $F \subseteq \im(j\colon H_2(W;\ft^k_{\a}) \to H_2(W, M_K;\ft^k_{\a})$.
Note that in the case that $W$ is a slice disc exterior, $F=0$, $H_2(W, M_K;\ft^k_{\a})\cong TH_2(W, M_K;\ft^k_{\a})$ \cite[Lemma~4]{Casson-Gordon:1986-1} ~\cite[Corollary~4.2]{Friedl-Powell:2010-1}, and the result follows.

In the general case that $W$ is a 2-solution with $H_2(W;\Z) \neq 0$, we extend the argument of \cite[Lemma~4.5]{Cochran-Orr-Teichner:1999-1}, with $n=2$, $\mathcal{R}= \ft$ and $\mathcal{K}= \mathbb{F}(t)$ in the notation of that lemma.  Again, we want to show that
\[TH_2(W, M_K;\ft^k_{\a}) \xrightarrow{\delta|_T} TH_1(M_K;\ft^k_{\a}) \xrightarrow{i|_T} TH_1(W;\ft^k_{\a})\]
is exact.  This is by now a standard line of argument (a similar argument appears in \cite[Lemma~5.10]{Cochran-Harvey-Leidy:2009-1}, for example), but for the convenience of the reader we give the details here.

By the hypothesis that $W$ is a 2-solution, the inclusion induced map $H_1(M_K;\Z) \to H_1(W;\Z)$ is an isomorphism.  Also, there are embedded surfaces $L_1,\dots,L_n,D_1,\dots,D_n \subset W$  forming a basis for $H_2(W;\Z)$ with geometric intersection numbers $L_i \cdot L_j =0 = D_i \cdot D_j$, $L_i \cdot D_j= \delta_{ij}$, and such that $\pi_1(L_i) \subset \pi_1(W)^{(2)}$ and $\pi_1(D_j) \subset \pi_1(W)^{(2)}$ for all $i,j=1,\dots,n$.

We have that $H_i(W,M_K;\Z)=0$ for $i=0,1$.  By \cite{Friedl-Powell:2010-1}, which relies on Strebel~\cite{Strebel:1974-1}, together with the partial chain homotopy lifting argument of \cite[Proposition~2.10]{Cochran-Orr-Teichner:1999-1},  we conclude that $H_i(W,M_K;\mathbb{F}(t)^k)=0$ for $i=0,1$.
The chain homotopy lifting argument uses that $W$ is homotopy equivalent to a finite CW complex; this is true for all compact topological manifolds: Hanner~\cite[Theorem~3.3]{Han51} showed that every manifold is an ANR, and West~\cite{We77} showed that every compact ANR is homotopy equivalent to a finite CW complex.

We already saw that $\a(k,\chi)$ is acyclic (Lemma~\ref{lem:CGacyclic}), so that $H_i(M_K;\ftk)=0$ for all $i$.  By the long exact sequence of the pair $(W,M_K)$, we have that $H_1(W;\ftk)=0$.  Apply Poincar\'{e} duality, universal coefficients, and our observation above that $H_i(W,M_K;\mathbb{F}(t)^k)=0$ for $i=0,1$, to deduce that $H_i(W;\ftk)=0$ for $i=3,4$.  Therefore the Euler characteristic of $W$ with these coefficients, which is the Euler characteristic of a $k$-fold cover of $W$, is $k \cdot \chi(W) = \chi(W_k) = \dim_{\mathbb{F}(t)} H_2(W;\ftk)$.  We also use that $W$ is homotopy equivalent to a finite CW complex to define the Euler characteristic and to see that it is multiplicative under finite covers.  On the other hand $H_1(W;\Z) \cong H_0(W;\Z) \cong \Z$, $H_2(W;\Z) \cong \Z^{2n}$, and $H_i(W;\Z) =0$ otherwise, so that $\chi(W)=2n$.  Therefore $\dim_{\mathbb{F}(t)} H_2(W;\ftk) = 2nk$.

The final part of this proof follows the proof of \cite[Lemma~4.5]{Cochran-Orr-Teichner:1999-1} very closely.
Consider the diagram:
 \[\xymatrix{H_2(W;\ft^k_{\a}) \ar[r]^{j_*} \ar[dr]_{\lambda_W} & H_2(W,M_K;\ft^k_{\a})  \ar[r]^{\delta} \ar[d]^{PD^{-1}} & H_1(M_K;\ft^k_{\a}) \\
 & \Hom_{\ft}(H_2(W;\ft^k_{\a}),\ft). &
 }\]
The map labelled $PD^{-1}$ is really the composition of the inverse of Poincar\'{e} duality and the Kronecker evaluation map.
The diagonal map labelled $\lambda_W$ is the adjoint of the intersection form on $H_2(W;\ft^k_{\a})$.

Now, by the hypotheses on $\pi_1(L_i)$ and $\pi_1(D_i)$, and the fact that $\a(k,\chi)$ vanishes on $\pi_1(W)^{(2)}$, the surfaces $L_i$, $D_j$ lift to embedded surfaces, which represent homology classes that we collect into the set
\[\mathcal{B} := \Big\{[L_i^j],[D_i^j] \in H_2(W;\ft^k_{\a}) \mid i=1,\dots n,\, j=1,\dots,k \Big\}.\]
The intersection form of $W$ with $\ft^k_{\a}$ coefficients, restricted to these elements is a direct sum of hyperbolic forms $\begin{pmatrix}
   0 & 1 \\ 1 & 0
 \end{pmatrix}$, since the geometric intersections lift to the cover.  Each surface in the collection has exactly one dual and intersects trivially with every other element of $\mathcal{B}$, from which it follows that the set $\mathcal{B}$ is linearly independent over $\ft$. To see this, suppose that $\sum a_i \beta_i=0$, where $a_i \in \ft$ and $\beta_i \in \mathcal{B}$. Then intersect with the dual of $\beta_j$ to obtain $a_j =0$.
 Note that since $\mathbb{F}(t)$ is a field, $\mathcal{B}$ induces a basis for $H_2(W;\ftk)$.

 Let $i_* \colon \ft^{2nk} \to H_2(W;\ft^k_{\a})$ be a map sending the basis elements of $\ft^{2nk}$ to the elements of $\mathcal{B}$.  The cokernel $C$ of this map $i_*$ is $\ft$-torsion, since these elements are a basis over $\F(t)$.  Therefore $\Hom_{\ft}(C,\ft) =0$, and so \[i^* \colon \Hom_{\ft}(H^2(W;\ft^k_{\a}),\ft) \to \Hom_{\ft}(\ft^{2nk},\ft)\]
 is injective.

By the above description of $\lambda_W$ on the elements of $\mathcal{B}$, the composition
 \begin{align*}
 \ft^{2nk} \xrightarrow{i_*} H_2(W;\ft^k_{\a}) \xrightarrow{\lambda_W} \Hom_{\ft}(&H_2(W;\ft^k_{\a}),\ft) \\ &\xrightarrow{i^*} \Hom_{\ft}(\ft^{2nk},\ft)
 \end{align*}
 is represented by a block sum of hyperbolic matrices $\begin{pmatrix}
   0 & 1 \\ 1 & 0
 \end{pmatrix}$.  Since this matrix is invertible, the composition is invertible, and so $i^*$ is surjective.  Therefore $i^*$ is an isomorphism, from which it follows that $\lambda_W$ is surjective.

Let $P:= \ker(H_1(M_K;\ft^k_{\a}) \to H_1(W;\ft^k_{\a})$. Since $H_1(M_K;\ft^k_{\a})$ is $\ft$-torsion, it suffices to show that the element $p \in P$ lies in the image of $TH_2(W,M_K;\ft^k_{\a}) \subset H_2(W,M_K;\ft^k_{\a})$.  The map $PD^{-1}$ is a surjection between modules with the same rank over $\F(t)$, and therefore the kernel of $PD^{-1}$ consists of $\ft$-torsion.

Choose $x \in H_2(W,M_K;\ft^k_{\a})$ with $\delta x =p$.  Let $y \in \lambda_W^{-1}(PD^{-1}(x))$; such a $y$ exists since $\lambda_W$ is surjective.  Then $x-j_*(y)$ lies in the kernel of $PD^{-1}$ and so is a torsion element, and we have with $\delta(x-j_*(y))=p$ since $\delta \circ j_*(y) =0$. Thus $p$ lies in the image of $TH_2(W,M_K;\ft^k_{\a})$ as required.  This completes the proof that (\ref{torsionexact}) is exact in the case that $W$ is a 2-solution.
\end{proof}

\subsection{A slice obstruction theorem}\label{section:slice-obstruction-theorem}

The following obstruction theorem is now an immediate consequence of Proposition \ref{prop:metabolic}
and Proposition \ref{prop:extend}.  Rather than just give a slice obstruction theorem, we present the solvable filtration refinement.  Recall that for a character $\chi \colon H_1(\Sigma_k;\Z)\to \Z_{q^a}$, we let $\chi_b$ be the composition of $\chi$ with the natural inclusion $\Z_{q^a} \hookrightarrow \Z_{q^b}$, for any integer $b \geq a$.  As above use $\mathbb{F}$ to denote the cyclotomic field $\Q(\zeta_{q^b})$.

\begin{theorem}\label{thm:slice}
Let $K$ be an oriented $2$-solvable knot with a $2$-solution $W$.
 Then for any prime power $k$ there exists a metaboliser $P$ of $\lambda_k$
such that for any prime power $q^a$, and any nontrivial
 character $\chi\colon H_1(\Sigma_k;\Z)\to \Z_{q^a}$ vanishing on $P$, we have some $b \geq a$ such that the twisted Blanchfield pairing $\Bl(\a(k,\chi_b))$ is metabolic with metaboliser \[\ker\left( TH_1(M_K;\ft^k_{\a}) \to TH_1(W;\ft^k_{\a})\right).\]
\end{theorem}

As advertised above, in light of Lemma \ref{lem:alexnorm} this theorem implies the slice obstruction theorem given by Kirk and Livingston in terms
of twisted Alexander polynomials~\cite[Theorem~6.2]{Kirk-Livingston:1999-2} (see also \cite{HKL08}), as long as $q \neq 2$ and so we can combine Theorem~\ref{thm:slice} and Lemma~\ref{lem:alexnorm} with the rather useful~\cite[Lemma~6.4]{Kirk-Livingston:1999-2}. This latter lemma shows that one is obligated only to check that a twisted Alexander polynomial does not factorise for an initial representation $\chi_a$, in order to conclude that the twisted polynomial does not factorise for all possible extensions from $\chi_a$ to $\chi_b$.

\section{Infection and solvability}\label{section:infection-and-solvability}

Now suppose that $\eta$ is a simple closed curve in $X_K$ that bounds an embedded disc in $S^3$, so that for any other knot $J$ the satellite knot $K_{\eta}(J)$ is defined.  Recall that the satellite knot $K_\eta(J)$ is the knot that makes the following a homeomorphism of pairs:
\[((S^3 \sm \nu \eta) \cup_{\partial \mathrm{cl}(\nu \eta) \cong \partial(S^3 \sm \nu J)} (S^3 \sm \nu J),K) \xrightarrow{\cong} (S^3,K_{\eta}(J)), \]
where the gluing identifies the meridian of $\eta$ with the zero-framed longitude of $J$, and vice-versa.  Note that up to isotopy there is a unique orientation preserving homeomorphism $S^3 \to S^3$, so $K_{\eta}(J)$ is well-defined.

 The curve $\eta$ determines a conjugacy class of elements in the knot group $\pi_1(X_K)$, and hence a conjugacy class of elements in $\pi_1(M_K)$. We will frequently abuse notation and write e.g.\ $\eta \in \pi_1(M_K)^{(2)}$ to mean that some (equivalently every) representative of the conjugacy class of elements of $\pi_1(M_K)$ corresponding to $\eta$ lies in the second derived subgroup $\pi_1(M_K)^{(2)}$. Whenever $W$ is a 4-manifold with $\partial W= M_K$, we also abusively refer to the image of  $[\eta]$ in $\pi_1(W)$ by $\eta$.

Recall from the beginning of Section~\ref{section:TBFs-and-knot-concordance} that $X_k(K)$ is the $k$-fold cyclic cover of the knot exterior $X_K$, and $M_k(K)$ is the $k$-fold cyclic cover of the zero-framed surgery $M_K$.  Also we write $\Sigma_k(K)$ for the $k$-fold cyclic branched cover of $S^3$ over~$K$.  The next proposition includes a somewhat novel use of mixed coefficient derived series.

\begin{prop}\label{prop:keyproposition}
Let $K$ be an oriented knot in $S^3$ and let $\eta \in \pi_1(M_K)^{(2)}$.
Suppose that there is a $2$-solution $W$ for $K$ such that  for some prime $q$ and every $r \in \N$,  we have that $\eta \mapsto 0$ in $\pi_1(W)/ \pi_1(W)^{(3)}_{(\Q, \Z_{q^r}, \Q)}$.
Then for each prime power $k$, the image $P_k$ of $P_k':=\ker\{H_1(X_k(K)) \to H_1(M_k(K))\to H_1(W_k)\}$ in $H_1(\Sigma_k)$ is a metaboliser for $\lambda_k$ satisfying the following:

Suppose $a \in \mathbb{N}$, and $\chi \colon H_1(\Sigma_k(K)) \to \mathbb{Z}_{q^a}$ has $\chi|_{P_k}=0$.
Then there is  an integer $b\geq a$,
such that if we
let $\alpha(k,\chi_b)$ be defined as above, then
\[ \Bl_{M_K}^{\alpha(k, \chi_b)}([v \otimes \widetilde{\eta}], [v \otimes \widetilde{\eta}])=0 \text{ in } \Q(\zeta_{q^b})(t)/ \Q(\zeta_{q^b})\tpm\]
where $v$ is any vector in $\Q(\zeta_{q^b}) \tpm^k$ and $\widetilde{\eta}$ is any lift of $\eta$ to the cover of $M_K$ corresponding to $\ker(\alpha(k,\chi_b))$.
\end{prop}

We will want to apply the contrapositive of this result.
Note that in practice it is often possible to show that $\Bl_{M_K}^{\alpha(k, \chi_a)}([v \otimes \widetilde{\eta}], [v \otimes \widetilde{\eta}]) \neq 0$ in $\C(t)/ \C[t^{\pm 1}]$,
and hence avoid explicitly working with the arbitrarily large degree extensions $\Q(\zeta_{q^b})$.

\begin{proof}
By Proposition \ref{prop:extend}, there exists some $b \geq a$ such that $\alpha(k,\chi_b)$ extends to a map $\pi_1(W) \to \Z \ltimes \Z_{q^b} \to \GL(\Q(\zeta_{q^b})\tpm^k)$.

Note that $\eta \in \pi_1(M)^{(2)}$ implies that $\eta \in \pi_1(W)^{(2)} \subseteq \pi_1(W)^{(2)}_{(\Q, \Z_{q^b})}$.
 Since $\eta \in \pi_1(W)^{(3)}_{(\Q,\Z_{q^b},\Q)}$, there is an integer $s \in \N$ such that $\eta^s$ belongs to the commutator subgroup of the second $(\Q, \Z_{q^r}, \Q)$-local derived subgroup $\Big[\pi_1(W)^{(2)}_{(\Q,\Z_{q^b})}, \pi_1(W)^{(2)}_{(\Q,\Z_{q^b})}\Big]$.
So we can write
\[ \eta^s= \prod_{i=1}^{\ell} [\eta_i^1, \eta_i^2]  \text{ for some } \eta_i^j \in  \pi_1(W)^{(2)}_{(\Q,\Z_{q^b})}.
\]

Since $\big(\Z \ltimes \Z_{q^b}\big)^{(1)}_{(\Q)}=1 \times \Z_{q^b}$ is an abelian group, we have that
\[
\big(\Z \ltimes \Z_{q^b}\big)^{(2)}_{(\Q, \Z_{q^b})}= \ker \left[
\left( 1 \times \Z_{q^b} \right)  \to \left( 1 \times \Z_{q^b} \right) \to \left(\left( 1 \times \Z_{q^b} \right) \otimes \Z_{q^b}\right)
\right]=1.
\]

Moreover $\alpha(k,\chi_b)\colon \pi_1(W) \to GL(\Q(\zeta_{q^b}) \tpm, k)$ factors through $\Z \ltimes \Z_{q^b}$,  and so $\alpha(k, \chi_b)(\eta_i^j)=\Id$ for all $i=1, \dots, \ell$ and $j=1,2$. Therefore each $\eta_i^j$ can be lifted to an element of $\pi_1(W_{\alpha(k,\chi_b)})$, where $W_{\alpha(k,\chi_b)}$ is the cover of $W$ corresponding to $\ker(\alpha(k,\chi_b))$.
So  $\widetilde{\eta^s}$ is a commutator of curves that lift to $W_{\alpha(k,\chi_b)}$, and hence $[\widetilde{\eta^s}]=0$ as an element of $H_1(W_{\alpha(k,\chi_b)}, \Z)$. Thus there is some 2-chain $A$ in $C_2(W_{\alpha(k,\chi_b)})$ with boundary $\widetilde{\eta^s}$.
Now recall that
\begin{align*}
H_1(W, \Q(\zeta_{q^b}) \tpm ^k_{\a(k, \chi_b)})&\cong H_1\big(\Q(\zeta_{q^b}) \tpm^k_{\a(k, \chi_b)} \otimes_{\Z[\pi_1(W)]} C_*(\widetilde{W})\big)\\
&\cong H_1\big(\Q(\zeta_{q^b}) \tpm^k_{\a(k, \chi_b)} \otimes_{\Z[\Z \ltimes \Z_{q^b}]} C_*(W_{\alpha(k, \chi_b)})\big).
\end{align*}

It follows that for any $v \in \Q(\zeta_{q^b}) \tpm^k$, the 2-chain $v \otimes A$ has boundary $v \otimes \widetilde{\eta^s}$,
 and hence that $[v \otimes \widetilde{\eta^s}] = s[v \otimes \widetilde{\eta}] =0$
in  $H_1(W, \Q(\zeta_{q^b}) \tpm^k_{\a(k, \chi_b)})$. However, $s$ is invertible in $\Q(\zeta_{q^b})$, and so we must have $[v \otimes \widetilde{\eta}]=0$ in $H_1(W, \Q(\zeta_{q^b}) \tpm^k_{\a(k, \chi_b)})$ as well.
By Proposition~\ref{prop:metabolic}, we obtain as desired that
\[\Bl_{M_K}^{\alpha(k, \chi_b)}\big( [v \otimes \widetilde{\eta}], [v \otimes \widetilde{\eta}]\big) = 0. \qedhere\]
\end{proof}

The previous key proposition allows us to use our ability to compute twisted Blanchfield pairings to show that elements $\eta \in \pi_1(M_K)^{(2)}$ survive in the fundamental group of any $2$-solution,
 in particular in the fundamental group of any slice disc exterior.  Next we apply this to obstruct the existence of $2.5$-solutions for satellite knots constructed using~$\eta$.

In the upcoming proof we will use the Cheeger-Gromov $L^{(2)}$ Von Neumann  $\rho$-invariant \cite{Cheeger-Gromov:1985-1}, \cite{Chang-Weinberger:2003-1}, \cite{Cochran-Teichner:2003-1}, \cite{Cochran-Harvey-Leidy:2009-1}, \cite{Cha-Orr:2009-01}.  Given a compact, closed 3-manifold $M$, a group $\Gamma$, and a homomorphism $\phi \colon \pi_1(M) \to \Gamma$, there is defined the Cheeger-Gromov $L^{(2)}$ Von Neumann $\rho$-invariant \[\rho^{(2)}(M,\phi) \in \R.\]  We refer to \cite{Cochran-Teichner:2003-1},\cite{Cha:2016-CG-bounds} for the precise definition.
The key property, which for the purpose of this paper may be taken as a definition, is its interpretation as a signature defect, as we now describe.
Let $W$ be a compact connected 4-manifold with $\partial W = M$, such that there is a group $\Lambda$, an embedding $\Gamma \hookrightarrow \Lambda$, and a representation $\Phi \colon \pi_1(W) \to \Lambda$ such that
\[\xymatrix{\pi_1(M) \ar[r] \ar[d]^-{\phi} & \pi_1(W) \ar[d]^-{\Phi} \\ \Gamma \ar @{^{(}->}[r] & \Lambda}\]
commutes.
Let $\mathcal{N}\Lambda$ be the Von Neumann algebra of $\Lambda$, that is the completion of $\mathbb{C}\Lambda$ with respect to pointwise convergence in the space of bounded operators on $\ell^2\Lambda$, the square summable elements of $\mathbb{C}\Lambda$.  There is an $L^{(2)}$-signature $\sigma^{(2)}(W,\Phi) \in \R$, the signature of the intersection form
\[\lambda_{\mathcal{N}\Lambda}\colon  H_2(W;\mathcal{N}\Lambda) \times H_2(W;\mathcal{N}\Lambda) \to \mathcal{N}\Lambda.\]
Let $\sigma(W) \in \Z$ be the signature of the ordinary intersection form
\[\lambda_{\R} \colon  H_2(W;\R) \times H_2(W;\R) \to \R.\]
Then we have
\[\rho^{(2)}(M,\phi)= \sigma^{(2)}(W,\Phi) - \sigma(W) \in \R.\]

\begin{thm}\label{thm:sliceobstruction}
Let $R$ be a slice knot and let $\eta \in \pi_1(X_R)^{(2)}$.
Suppose that there is some prime power $k$ such that for each metaboliser $P$ for the linking form $\lambda^k_R$ the following condition holds:

There is some character $\chi_P \colon H_1(\Sigma_k(R)) \to \mathbb{Z}_{q^a}$, for $q$ prime and $a>0$, such that $\chi_P$ vanishes on $P$ and for all $b\geq a$ we have
\[ \Bl_{M_R}^{\alpha(k, \chi_P)}([v \otimes \widetilde{\eta}], [v \otimes \widetilde{\eta}]) \neq 0 \text{ in } \mathbb{Q}(\zeta_{q^b})(t)/ \mathbb{Q}(\zeta_{q^b})\tpm\]
where $v$ is some vector in $\Q(\zeta_{q^b}) \tpm^k$ and $\widetilde{\eta}$ is any lift of $\eta$ to the cover of $M_K$ corresponding to $\ker(\alpha(k,\chi_b))$.

Then there is a constant $C_R>0$, depending only on the knot $R$, such that if $J$ is a knot with Arf invariant 0 and $\left| \int_{S^1}\sigma_J(\omega)\, d \omega \right|>C_R$, then $K := R_{\eta}(J)$ is 2-solvable but not $2.5$-solvable.
\end{thm}

In fact, it was shown in \cite{Cha:2016-CG-bounds} that one can take $C_R$ equal to $10^8$ times $c(R)$, the crossing number of $R$.
For $J$ the trefoil, $\left| \int_{S^1}\sigma_J(\omega)\, d \omega \right| = 4/3$, so a connected sum of $10^8\cdot c(R)$ trefoils suffices.

\begin{proof}
First note that since $\eta \in \pi_1(X_R)^{(2)}$, and $J$ has vanishing Arf invariant, \cite[Proposition 3.1]{Cochran-Orr-Teichner:2002-1} tells us that the satellite knot $K= R_{\eta}(J)$ is $2$-solvable.
Now suppose that $K$ is $2.5$-solvable with $2.5$-solution $W$.
Let $k$ be a prime power as in the statement of the theorem, and let $P_K$ be the image of  \[\ker \left(H_1(X_k(K)) \to H_1(M_k(K)) \to H_1(W_k)\right)\] in $H_1(\Sigma_k(K))$. By Lemma~\ref{lem:linkmetabolic}, $P_K$ is a metaboliser for the torsion linking form $\lambda_k^K$ on $H_1(\Sigma_k(K))$.

Let $f \colon X_J \to X_U$ be the usual degree one map.
 Observe that since $\eta \in \pi_1(X_R)^{(2)}$ we have a decomposition
\[\Sigma_k(K)= \Big(\Sigma_k(R) \sm \tmcup{i=1}{k} \nu(\widetilde{\eta_i})\Big) \cup \tmcup{i=1}{k} X_J,\]
where $\{\widetilde{\eta_i}\}$ are the $k$ lifts of $\eta$ to $\Sigma_k(R)$. The map $f$ induces a covering transformation invariant degree one map $f_k\colon \Sigma_k(K) \to \Sigma_k(R)$. Since $\eta \in \pi_1(X_R)^{(2)}$, we have $[\widetilde{\eta_i}]=0$ in $H_1(\Sigma_k(R))$ and so $f_k$ induces  covering transformation invariant isomorphisms $H_*(\Sigma_k(K))\to H_*(\Sigma_k(R))$. We now argue that these isomorphisms also induce an isomorphism between the torsion linking forms.

Observe that by functoriality of the cap and cup products, the following diagram commutes.
\[\xymatrix @C+0.5cm{H^2(\Sigma_k(K)) \ar[r]^-{\PD_{\Sigma_k(K)}}  & H_1(\Sigma_k(K)) \ar[d]^-{(f_k)_*}  \\
H^2(\Sigma_k(R))\ar[r]^-{\PD_{\Sigma_k(R)}} \ar[u]^-{(f_k)^*} & H_1(\Sigma_k(R))
}\]
That is, we have $(f_k)_* \circ PD_{\Sigma_k(K)} \circ (f_k)^*= PD_{\Sigma_k(R)}$.
Since $(f_k)_*$ (and hence, by the above diagram, $(f_k)^*$) is an isomorphism, it follows that $(f_k)^* \circ \PD_{\Sigma_k(R)}^{-1} \circ (f_k)_* = \PD_{\Sigma_k(K)}^{-1}$. That is, the left square in the following diagram commutes.
\[\xymatrix @C+0.5cm{ H_1(\Sigma_k(K)) \ar[r]^-{\PD_{\Sigma_k(K)}^{-1}} \ar[d]_-{(f_k)_*}  & H^2(\Sigma_k(K)) \ar[r]  &H_1(\Sigma_k(K))^{\wedge}  \\
H_1(\Sigma_k(R)) \ar[r]^-{\PD_{\Sigma_k(R)}^{-1}}  &H^2(\Sigma_k(R)) \ar[r] \ar[u]_-{(f_k)^*} & H_1(\Sigma_k(R))^{\wedge} \ar[u]_-{(f_k)^{\wedge}}
}\]
The right square commutes by general principles, and
the full composition from left to right in the top and bottom rows gives the torsion linking forms $\lambda_k^K$ and $\lambda_k^R$, respectively. So we have our desired claim that $f_k$ induces an isomorphism of linking forms. In particular,  $(f_k)_*$ identifies $P_K$  with a metaboliser $P_R \leq H_1(\Sigma_k(R))$ for $\lambda_k^R$.

Let $\chi_R\colon H_1(\Sigma_k(R)) \to \Z_{q^a}$ be as in the assumptions of the theorem and define $\chi_K := \chi_R \circ (f_k)_*$. Observe that $\chi_K|_{P_K}=0$. Now let $b \geq a$ be an arbitrary integer, and let $\chi_K'$ and $\chi_R'$ be the corresponding extensions.
Note that $f$  extends by the identity on $X_R \sm \nu(\eta)$ to a degree one map $g \colon M_K \to M_R$.  The map $g$ induces a map between the pair of long exact sequences corresponding to the decompositions $M_{K} = \left(M_{R} \sm \nu(\eta) \right) \cup X_J$ and $M_{R}=\left(M_R \sm \nu(\eta)\right) \cup \nu(\eta)$:
\[\xymatrix @C-0.3cm {H_i(\partial(\nu (\eta))) \ar[r] \ar[d]^{\Id}_{\cong} & H_i(M_R\sm \nu(\eta)) \oplus H_i(X_J) \ar[r] \ar[d]_{\cong}^{\Id \oplus f_*} & H_i(M_K) \ar[r] \ar[d]^{g_*} & H_{i-1}(\partial(\nu (\eta))) \ar[r] \ar[d]^{\Id}_{\cong}  & \cdots \\
H_i(\partial(\nu (\eta))) \ar[r] & H_i(M_R\sm \nu(\eta)) \oplus H_i(\nu (\eta)) \ar[r]  & H_i(M_R) \ar[r] &  H_{i-1}(\partial(\nu (\eta))) \ar[r]  & \cdots
}\]
 where all homology is taken with twisted coefficients in $\Q(\zeta_{q^b})\tpm$, defined by $\alpha(k,\chi_K'), \alpha(k,\chi_R')$ and their composition with appropriate inclusion induced maps. Since $\mu_J= \lambda_\eta \in \pi_1(X_R)^{(2)}$ and $\alpha(k,\chi_K')$ factor through a map to a metabelian group, the composition $\pi_1(X_J) \to \pi_1(M_K) \to \GL(\Q(\zeta_{q^b})\tpm^k)$ is the zero map.
In particular, the twisted homology of $X_J$ is isomorphic to that of $X_U$ via $f_*$. Apply the five lemma to see that \[g_*\colon H_*(M_K, \Q(\zeta_{q^b})\tpm^k_{\a}) \to H_*(M_R, \Q(\zeta_{q^b})\tpm^k_{\a})\] is also an isomorphism.

Arguments directly analogous to those used above to show that $(f_k)_*\colon H_*(\Sigma_k(K)) \to H_*(\Sigma_k(R))$ induces an isomorphism between the torsion linking forms can now be applied to show that $g_*$ induces an isomorphism between $\Bl_K^{\alpha(k,\chi_K')}$ and $\Bl_R^{\alpha(k,\chi_R')}$.
 Finally, since $g$ is the identity map on $X_R \sm \nu(\eta)$ we have that
\begin{align*}
\Bl_K^{\alpha(k,\chi_K')}\big([v \otimes \widetilde{\lambda_\eta}],[v \otimes \widetilde{\lambda_\eta}]\big)&=
\Bl_R^{\alpha(k,\chi_R')}\big(g_*([v \otimes \widetilde{\lambda_\eta}]), g_*([v \otimes \widetilde{\lambda_\eta}])\big) \\
&= \Bl_R^{\alpha(k,\chi_R')}\big([v \otimes \widetilde{\eta}],[v \otimes \widetilde{\eta}]\big)
 \neq 0 \in \frac{\Q(\zeta_{q^b})(t)}{\Q(\zeta_{q^b})\tpm}.\end{align*}
 By Proposition~\ref{prop:keyproposition} it follows that, for some $r \in \N$, $\mu_J= \lambda_\eta$ maps nontrivially to $\pi_1(W)/ \pi_1(W)^{(3)}_{(\Q, \Z_{q^r}, \Q)}$.

 Now let $\phi \colon \pi_1(M_K) \to \Lambda:= \pi_1(W)/ \pi_1(W)^{(3)}_{(\Q, \Z_{q^r}, \Q)}$. Observe that $\Lambda$ is amenable and in Strebel's class $D(\Z_q)$ by~\cite[Lemma~4.3]{Cha:2014-1}.
Since $\phi$ extends over $\pi_1(W)$ (by definition it factors through $\pi_1(W)$), the amenable signature theorem \cite[Theorem~3.2]{Cha:2014-1} (which was based on \cite{Cha-Orr:2009-01}) tells us that $\rho^{(2)}(M_K, \phi)=0$.  The hypotheses of the amenable signature theorem require that $W$ be a $2.5$-solution.
However, the argument of \cite[Section~4.4]{Cha:2012-1}\footnote{The argument of \cite[Section~4.4]{Cha:2012-1} is based on \cite[Lemma~2.3]{Cochran-Harvey-Leidy:2009-1} and \cite[Proposition~3.2]{Cochran-Orr-Teichner:2002-1}, but these references use slightly more restricted coefficient systems; the argument is unchanged for mixed coefficient derived series.},
 implies that \[\rho^{(2)}(M_K,\phi)= \rho^{(2)}(M_R, \phi_R)+ \rho^{(2)}(M_J, \phi_J),\] where $\phi_R$ and $\phi_J$ are the unique extensions of the representations $\phi|_{\pi_1(M_R \sm \nu(\eta))}$ and $\phi|_{\pi_1(X_J)}$ to $\pi_1(M_R)$ and $\pi_1(M_J)$, respectively.
We therefore have an equality
\[|\rho^{(2)}(M_R, \phi_R)|= |\rho^{(2)}(M_J, \phi_J)|.\]
By \cite{Cheeger-Gromov:1985-1}, there exists a constant $C_R>0$ depending only on $R$ such that if $\psi$ is any representation $\psi\colon \pi_1(M_R) \to \Gamma$ then $|\rho^{(2)}(M_R, \psi)| < C_R$.  In order to obtain a contradiction, and deduce that $K$ is not $2.5$ solvable, it therefore suffices to show that $\rho^{(2)}(M_J, \phi_J)=  \int_{S^1}\sigma_J(\omega)\, d \omega$.

Since $\pi_1(M_J)$ is generated by meridians of $J$, all of which are identified with longitudes of $\eta$ and hence lie in $\pi_1(M_K)^{(2)}\subseteq \pi_1(M_K)^{(2)}_{(\Q, \Z_{q^r})}$, the map $\phi_J$ maps $\pi_1(M_J)$ into $\pi_1(W)^{(2)}_{(\Q, \Z_{q^r})}/ \pi_1(W)^{(3)}_{(\Q, \Z_{q^r}, \Q)}$, which is a torsion-free abelian group.
It follows that $\phi_J$ is either trivial or a maps onto a copy of $\Z$ in $\pi_1(W)^{(2)}_{(\Q, \Z_{q^r})}/ \pi_1(W)^{(3)}_{(\Q, \Z_{q^r}, \Q)}$.  But $\phi_J$ is nontrivial since  $\mu_J= \lambda_{\eta}$ does not lie in $\pi_1(W)^{(3)}_{(\Q, \Z_{q^r}, \Q)}$. By $L^{(2)}$-induction we therefore have our desired result that \[\rho^{(2)}(M_J, \phi_J)= \rho^{(2)}(M_J, \ab)= \int_{S^1}\sigma_J(\omega)\, d \omega\]
 by \cite[Propositions~2.3~and~2.4]{Cochran-Orr-Teichner:2002-1}, where $\ab \colon \pi_1(M_J) \to \Z$ is the abelianisation homomorphism.
\end{proof}

\section{Examples of non-slice knots}\label{section:examples}

Theorem~\ref{thm:sliceobstruction} gives a straightforward method to show that
for $R$ a ribbon knot and $\eta \in \pi_1(M_R)^{(2)}$,
appropriately large infections (in terms of the $\rho^{(2)}$-invariants of the infection knots) on $(R,\eta)$ are not slice or even $2.5$-solvable. We give three examples illustrating this.  Note that in each example, one could instead choose any curve $\eta'$ that is an unknot in $S^3$ with  $[\eta']=[\eta] \in \pi_1(M_R)^{(2)}/ \pi_1(M_R)^{(3)}$ and obtain the same result on the non-sliceness of $R_{\eta'}(J)$. Our examples all involve representations associated to the double branched cover, however this method works equally well for metabelian representations associated to higher prime power order branched covers.

Our first example is a small crossing number prime ribbon knot, chosen without any special prejudice from the knot tables. We go through this example in some detail.  The Maple program available on our websites contains the data working through this example.

\begin{example}\label{example:one}\textit{The knot $R_1=8_8$.}
Following the conventions of Section~\ref{section:chaincomplex} as indicated in Figure~\ref{Fig:88example}, we have that $\eta= [g_5^{-1}g_1, g_7^{-1} g_3] \in \pi=\pi_1(M_{R_1})$. (We use the convention that $[a,b]= aba^{-1}b^{-1}$.)  In particular, note that $\eta$ is in $\pi^{(2)}$.
\begin{figure}[h]
  \begin{center}
  \includegraphics[height=5cm]{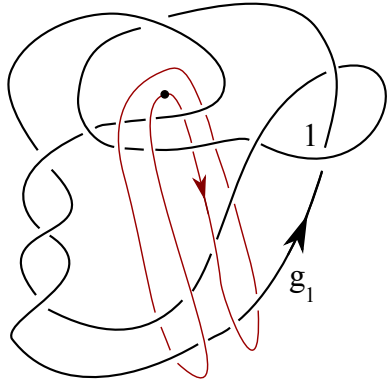}
  \caption{The knot $R_1= 8_8$.}
  \label{Fig:88example}
  \end{center}
\end{figure}
We have that $H_1(\Sigma_2(K))\cong \Z_{25}$, and hence that any character $\chi$ to $\Z_5$  must vanish on the unique metaboliser for $\lambda_2$. Note that all such onto characters will be nonzero multiples of each other, and therefore induce the same metabelian covers. So we choose such an onto map at random. Let $V= \Q(\zeta_{5})\tpm^2$.
Note that the chain group $Y_1 = V \otimes_{\Z[\pi]} C_1(M_{R_1}, \Z[\pi])$ is a free $\Q(\zeta_{5})\tpm$-module of rank 16, with basis given by $\{[1,0] \otimes \widetilde{g_i}, [0,1] \otimes \widetilde{g_i}\mid i=1,\dots,8\}$, where $\widetilde{g_i}$ are the preferred lifts of $[g_i] \in C_1(M_{R_1})$ to the universal cover. It is then straightforward to compute that with respect to this basis, $[1,0] \, \otimes\, \widetilde{\eta}$ is given by
\[
\left[
\begin{array}{cccccccccccccccc}
0&\zeta_5^2-\zeta_5&0&0   &0&1-\zeta_5^3&0&0  &0&-\zeta_5^2+\zeta_5 &0&0   &0 &-1+\zeta_5^3&0&0
\end{array}
 \right]
\]
Now we follow Sections~\ref{section:trotters-formulae} and~\ref{section:matrix-moves-computing-TBF} to compute $b_1:=\Bl_{M_{R_1}}^{\alpha(2, \chi)}([[1,0] \otimes \tilde{\eta}], [[1,0] \otimes \tilde{\eta}])$. We obtain that
$b_1=\frac{q(t)}{t^2-3t+1} \in \Q(\zeta_{5})(t)/ \Q(\zeta_{5})\tpm$
where
\begin{align*}
 q(t)=
\left( \frac{125}{2019061} \right)
& \left[ \left(14950863  +78888245 \zeta_5 - 11038450 \zeta_5^2-15959303 \zeta_5^3
\right) t \right.
\\
& \qquad +\left. \left(
-5703559-3021030\zeta_5+4192146 \zeta_5^2 +6076728 \zeta_5^3 \right)\right].
\end{align*}
Note that the degree of $q(t)$ is strictly less than the degree of $t^2-3t+1$, and so $b_1$ cannot equal 0 even in $\C (t)/ \C \tpm$.
\end{example}

The next example, originally due to \cite{Cochran-Orr-Teichner:1999-1}, was the first example of an algebraically slice knot with vanishing Casson-Gordon invariants that is nevertheless not slice, nor even 2.5-solvable. The intricate arguments that they use seem very difficult to apply to other knots, relying as they do on the fact that their pattern knot $R_2$ is fibred, together with an extremely involved analysis of higher order Alexander modules coming from the monodromy of $R_2$. Our proof is simpler and applies much more generally.

\begin{example}\label{example:COT}\textit{The Cochran-Orr-Teichner example.}
We consider the example of \cite[Section~6]{Cochran-Orr-Teichner:1999-1}, as illustrated in \cite[Figure~6.5]{Cochran-Orr-Teichner:1999-1} and our Figure~\ref{Fig:cotexample}.
\begin{figure}[h]
  \begin{center}
  \includegraphics[height=5cm]{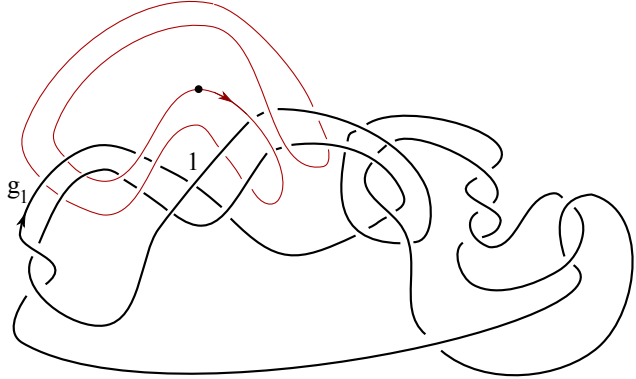}
  \caption{The ribbon knot $R_2$ with the `genetic modification' curve $\eta$ of Cochran-Orr-Teichner}
  \label{Fig:cotexample}
  \end{center}
\end{figure}
 Observe that $[\eta]=[g_7g_3^{-1}, g_1 g_4^{-1}] \in \pi_1(M_{R_2})^{(2)}$.
One can easily compute that $H_1(\Sigma_2(R)) \cong \Z_{25}$, and so there is a unique metaboliser $\Z_5 \cong P \leq H_1(\Sigma_2(R_2))$.
Let $\chi\colon H_1(\Sigma_2(R_2)) \to \Z_5$ be onto, and note that $\chi|_P=0$.
It now suffices to show that for some $v \in \Q(\zeta_5)\tpm^2$ and for any $b \geq 1$ we have that
\[b_2:=\Bl_{M_{R_2}}^{\alpha(2, \chi)}([v \otimes \tilde{\eta}], [v \otimes \tilde{\eta}])\neq 0 \in \Q(\zeta_{5^b})(t)/ \Q(\zeta_{5^b})\tpm.\]

We choose $v=[1,0]$ without any special prejudice. Computation as in Section~\ref{section:chaincomplex} gives us that $b_2= \frac{p(t)}{(t-1)^2}$ for a polynomial $p(t) \in \Q(\zeta_{5})\tpm$ with $p(1)= -6 -2(\zeta_5^2+\zeta_5^3)$. So since $p(1) \neq 0$, we have that $b_2 \neq 0 \in \C(t)/\C \tpm$, and hence our desired result.

\end{example}

Finally, we give a non-prime example for $R_3$, with the additional interesting feature that we can choose an infection curve that has interacts with only one prime factor of $R_3$. This example is also of interest in that it requires us to consider multiple metabolisers for the torsion linking form on $H_1(\Sigma_2(R_3))$.

\begin{example}\label{example:3} \textit{The square knot.}
Let $R_3= T_{2,3} \# - T_{2,3}$, and $\eta$ be as illustrated.

\begin{figure}[h]
  \begin{center}
  \includegraphics[height=5cm]{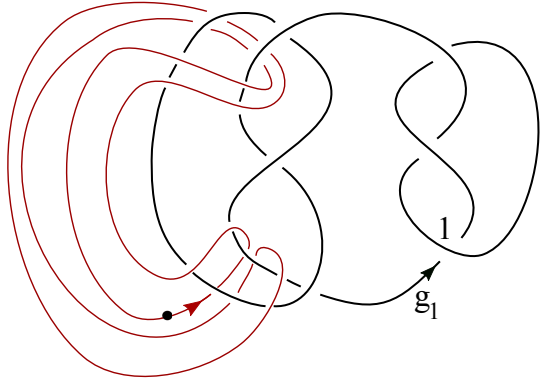}
  \caption{The square knot $R_3$}
  \label{Fig:squareknotexample}
  \end{center}
\end{figure}

Note that
$H_1(\Sigma_2(R_3)) \cong H_1(\Sigma_2(T_{2,3})) \oplus H_1(\Sigma_2(-T_{2,3}))=  \Z_3 \oplus \Z_3$ and that $\lambda_2^{R_3}= \lambda_2^{T_{2,3}} \oplus \lambda_2^{-T_{2,3}}$.
It is straightforward to check that there are two metabolisers for $\lambda_2^{R_3}$, which with respect to this decomposition are of the form $H_a = \langle (1,1) \rangle$ and $H_b= \langle (1,2) \rangle$.
 Let $\chi_a\colon H_1(\Sigma_2(R_3)) \to \Z_3$ be a nontrivial character vanishing on $H_a$, and $\chi_b$ a nontrivial character vanishing on $H_b$.
  In order to show that appropriate infections on $(R, \eta)$ are not slice (where as usual appropriate means infections by $J$ with sufficiently large $\left| \int_{S^1}\sigma_J(\omega)\, d \omega \right|$), it suffices to show that for some choice of $ v \in\Q(\zeta_{3})\tpm ^2$ and for any $s \geq 1$,
\[\Bl_{M_{R_3}}^{\alpha(2, \chi_a)}([v \otimes \tilde{\eta}], [v \otimes \tilde{\eta}]), \Bl_{M_{R_3}}^{\alpha(2, \chi_b)}([v \otimes \tilde{\eta}], [v \otimes \tilde{\eta}])
\neq 0 \in \Q(\zeta_{3^s})(t)/ \Q(\zeta_{3^s})\tpm.\]
As in the previous examples, with the help of the computer we are in fact able to show that these are both nonzero, even in $\C(t)/\C \tpm$.  We omit the details of the computation.
\end{example}

We note that there was nothing especially contrived about the knots $R_1$ and $R_3$ of our first and third examples, nor about the curves $\eta$ that we chose.  The advantage of our approach is that one obtains very explicit examples, without having to try very hard to choose the examples to fit our obstruction theory.   Of course the $\eta$ curves that we use have to sit in the right place in the derived series.  But since every algebraically slice knot is an infection by a string link on a slice knot~\cite[Proposition~1.7]{Cochran-Friedl-Teichner:2006-1}, the situation is somewhat generic.

\bibliographystyle{alpha}
\def\MR#1{}
\bibliography{research}

\end{document}